\documentclass{article}

\usepackage[utf8]{inputenc}
\usepackage{amsmath}
\usepackage{amsfonts}
\usepackage{amsthm}
\usepackage{myarticles}
\usepackage{bbm}
\usepackage{cases}
\usepackage{enumitem}
\usepackage{mathtools}
\usepackage{leftidx}
\usepackage[
backend=biber,
style=numeric,
sorting=nyt,
abbreviate=true,
natbib=true,
sortlocale=en_US,
isbn=false,
first inits=true,
url=false, 
doi=true,
eprint=false
]{biblatex}
\usepackage[]{todonotes}

\theoremstyle{Mio}
\newtheorem{prop}{Proposition}[section]
\newtheorem{lemma}{Lemma}[section]
\newtheorem{teo}{Theorem}[section]

\newtheorem{innercustomthm}{Assumption}
\newenvironment{Asmpt}[1]
{\renewcommand\theinnercustomthm{#1}\innercustomthm}
{\endinnercustomthm}

\theoremstyle{remark}
\newtheorem{rem}{Remark}[section]

\newtheorem{defi}{Definition}[section]
\date{}
\usepackage[hidelinks]{hyperref}

\title{Optimal switching problem for marked point process and systems of reflected BSDE}
\author{Nahuel Foresta\thanks{ 
Dipartimento di Matematica, Politecnico di Milano, via Bonardi 9, 20133 Milano, Italy. nahueltomas.foresta@polimi.it. The
 author was supported
by the Italian MIUR-PRIN 2015 ``Deterministic and stochastic evolution equations'' and
INDAM-GNAMPA.}
}

\begin{document}
    \maketitle
    \begin{abstract}
We formulate an optimal switching problem when the underlying filtration is generated by a marked point process and a Brownian motion. Each mode is characterized by a different compensator for the point process, and thus by a different probability $\prob^i$, which form a dominated family. To each strategy $\mathbf{a}$ of switching times and actions then corresponds a compensator and a probability $\prob^\mathbf{a}$, and the reward is calculated under this probability.

To solve this problem, we define and study a system of reflected BSDE where the obstacle for each equation depends on the solution to the others. The main assumption is that the point process is non explosive and quasi-left continuous. We prove wellposedness of this system through a Picard iteration method, and then use it to represent the optimal value function of the switching problem. We also obtain a comparison theorem for BSDE driven by marked point process and Brownian motion.\\[1cm]

    \end{abstract}

\section{Introduction}

In the last years, the field of optimal switching has received a lot of interest. One of the first descriptions can be found in \cite{brennan1985evaluating}, where a two mode switching model is proposed to model investments in the natural resource industry. Other cases in which optimal switching has been treated are \cite{brekke194optimalswit,carmona08optimalswithpricing,ludkovski2005optimal,tang1993switchimpulseviscosity} among others. In this kind of problems, a controller has at his disposal a certain number of modes in which a system can evolve, say $m$, and he can at any time switch from one mode to another. The rewards (or costs) that the agent obtains are different on each mode, and switching from one mode to another has a cost. Also, the stochastic dynamic of the underlying process may depend on the current mode.

Here we study the case where the underlying stochastic processes are a non-explosive marked point process and a Brownian motion. In particular the (random) compensator of the point process is assumed to be continuous in time, but not absolutely continuous with respect to the Lebesgue measure in time. Marked point processes have been extensively studied in the past (see \cite{jacod1975multivariate}, \cite{bremaud1981point}
or \cite{lastbrandt1995}). In particular \cite{bremaud1981point} discusses cases of impulse control for this kind of processes. Nevertheless, the BSDE approach to control problems for marked point process is rather recent (see \cite{confortola2013backward,confortola2014backward,foresta2017optimal})

Our problem is formulated in a non Markovian way. The peculiarity is how the current mode of evolution affects the dynamic of the point process: to each mode it corresponds a different probability, under which the law of the marked point process is different. This means that for each mode there is a different compensator. The probabilities form a dominated family with respect to some reference probability. Let us briefly explain how this problem is formulated, leaving the details to section \ref{sec:framework}. We start from one reference probability $\prob$ under which the point process $p$ has compensator $\phi_s(\omega,de)dA_s(\omega)$. For each of the $m$ modes we have running gains $f^i_t,g^i_t$ and terminal gain $\xi^i$, but we are also given a bounded predictable random field $\rho^i$ that induces a probability $\prob^i\ll\prob$. As usual, when switching from mode $i$ to mode $j$ there is a cost $C_t(i,j)$. 
The controller chooses a sequence of switching times and switching actions $\mathbf{a}=(\theta_k,\alpha_k)_{k\geq 0}$, which means that at time $\theta_k$ the system is switched to mode $\alpha_k$. This does not only modify the gains, but also the dynamic of the system in the following way. We define the quantities
$$
a_s=\sum_{k\geq 1}\alpha_{k-1}\ind_{(\theta_{k-1},\theta_k]}(s)\qquad \rho_s^{\mathbf{a}}=\sum_{k\geq 1}\ind_{(\theta_{k-1},\theta_{k}]}(s)\rho_s^{\alpha_{k-1}}(e).
$$
Then $\rho^\mathbf{a}$ introduces a new probability (through Girsanov transform) $\prob^\mathbf{a}\ll\prob$, under which the compensator of $p$ is $\rho_s^\mathbf{a}(e)\phi_s(de)dA_s$, and we use it to evaluate the gains. For a given strategy $\mathbf{a}$ starting at time $t$ form mode $i$ the expected gain is
$$
J(t,i,\mathbf{a})=\eval^{\mathbf{a}}\left[\left.\xi^{a_T}+\int_t^Tf_s^{a_s}dA_s+\int_t^Tg_s^{a_s}ds-\sum_{k\geq 1}C_{\theta_k}(\alpha_{k-1},\alpha_k)\right|\mathcal{F}_t\right].
$$
and the value function is the obtained as the essential supremum over all strategies. So we obtain what is in some cases called a ``weak" formulation

Since we are in a non Markovian framework, a natural tool to solve this problem is the use of BSDE. The first paper to use this technique is \cite{hamadene2007startingstopping}, where the authors solve the problem with two possible switching modes and the diffusive dynamic of the underlying process not dependent on the strategy. In order to do this they use the Snell envelope characterization of the solution processes, as well as a doubly reflected BSDE to prove the existence of said solution. This approach is later generalized to the case of any finite number of modes, first in \cite{djehiche2009finiteorizonmultipleswitching} using the Snell envelope characterization, and then both in \cite{hamadene2010switching} and \cite{hu2010multidimensionalBSDE}, where they establish existence of a system of reflected BSDE with interconnected obstacles. In the latter, they solve the switching problem in the case where the drift of the underlying process depends on the mode. The generalization to the case where the mode also affects the volatility, using a system of reflected BSDE, is done in \cite{elie2014controlledswitchingvolatility}. Another interesting result is the one contained in \cite{hamadene2013viscosityswitching}, where the BSDE system is linked to the viscosity solution of a system of PDE with connected obstacles associated to the switching problem.
There are other cases were jump type noises have also been done. One of the first is \cite{hamadene2007starting}, where a Poisson random measure is added to the two modes case. In \cite{hamadene2015systems} the case where the noises are Brownian motion and a Teugels martingale is studied, again with the help of a system of reflected BSDE.

In the present work, we extend in this direction by changing the nature of jumps that appear, but also by introducing this particular form of weak control. The system of reflected BSDE in our case is

\begin{equation}
\label{system_intro}
\begin{cases}
Y_t^i=\xi^i+\int_t^Tf_s^idA_s+\int_t^T g^i_sds+\int_t^T\int_E U_s^i(e)(\rho_s^i(e)-1)dA_s\\
\qquad -\int_t^TU_s^i(e)q(dsde)-\int_t^TZ_s^idW_s+K_T^i-K_t^i\\
Y_t^i\geq \max\limits_{j\in \mathbb{A}_i}(Y^j_t-C_t(i,j))\\
\int_0^T(Y_t^i- \max\limits_{j\in \mathbb{A}_i}(Y^j_t-C_t(i,j)))dK_t^i=0.
\end{cases}
\end{equation}

The solution $Y^i$ of the system represents the value function of the optimal switching problem. The BSDE are driven by both a Wiener process $W$ and the point process compensated random measure $q(dsde)=p(dsde)-\phi_s(de)dA_s$. In particular we make the fundamental assumption that the process $A$ appearing in the compensator $\phi_s(de)dA_s$ is continuous. The solution and the data lie in some kind of weighted $L^2$ spaces, where the weight is a function of the process $A$. See section \ref{sec:framework} for details. Equations of this type have already been studied in  \cite{confortola2013backward,confortola2014backward,confortola16LP,foresta2017optimal,} for the case with $A$ continuous, and in \cite{Bandini2015nonquasi,bandini2017optimal,cohen10generalcomparison,cohen12generalspaces} for the case with non continuous $A$. In particular the paper \cite{foresta2017optimal} studies a BSDE driven by a marked point process and a Brownian motion with a reflection, and thus poses the basis for this work. Nevertheless some results typical of the BSDE theory are missing, so we need to state and prove them. In particular we formulate and prove a comparison theorem for standard BSDE driven by marked point processes and Brownian motion.

The paper is organized as follows. In section \ref{sec:framework} we introduce precisely the formulation of the optimal switching problem and the system of reflected BSDE. We state the assumptions under which this is solved and we also recall briefly the properties of marked point processes. In the following section \ref{swit_existence} we study the existence of a solution to the system of reflected BSDE using a Picard iteration technique. For clarity of exposition, some of the accessory results used here, such as the comparison theorem, can be found in the appendixes. Lastly in section \ref{sec_swit_verification} we establish the link between the solution to system of reflected BSDE and the value function of the switching problem through a verification theorem. This also allows us to obtain uniqueness of the solution to the system.

\section{Framework and objectives}
\label{sec:framework}
In this section we give a brief introduction to the mathematical setting, define precisely  the problem we are trying to solve, state the hypotheses under which we work and introduce the system of reflected BSDE.

Consider a probability space $(\Omega,\mathcal{F},\prob)$, and a Borel\footnote{A borel space is a topological space homeomorphic to a Borel subset
	of a compact metric space (sometimes called Lusin space; we recall that every
	separable complete metric space is Borel)} space $(E,\mathcal{E})$. We are given a $d$-dimensional Wiener process $W$ and a marked point process $p$ with mark values in $E$, independent of $W$. A marked point process is a sequence $(T_n,\xi_n)_{n\geq 1}$ valued in $\mathbb{R}^+\times E$ such that $\prob$-a.s.
\begin{itemize}
\item $T_0=0$.
\item $T_n\leq T_{n+1} \forall n\geq 0$.
\item $T_n<\infty\Rightarrow T_n<T_{n+1} \forall n\geq 0$.
\end{itemize}
We will always assume the marked point process in the paper to be non-explosive, that is $T_n\rightarrow+\infty$ $\mathbb{P}$-a.s.
Another way of representing these processes is through the use of an integer random measure. To each marked point process we associate a random discrete measure $p$ on $((0,+\infty)\times E,\mathcal{B}((0,+\infty))\otimes\mathcal{E})$:
$$
p(\omega,D)=\sum_{n\geq 1}\ind_{(T_n(\omega),\xi_n(\omega))\in D}.
$$
Let $\mathbb{F}=(\mathcal{F}_t)_{t\geq 0}$ be the completed filtration generated by $p$ and $W$, which satisfies the usual conditions. Denote by $\mathcal{P}$
the $\sigma$-algebra of $\mathbb{F}$-predictable processes. To the marked point process is associated a unique predictable random measure $\nu$ on $\Omega\times\mathbb{R}^+\times E$,
called compensator, such that for all non-negative $\mathcal{P}\otimes\mathcal{E}$-measurable
process $C$ it holds that
$$
\evals{\int_0^{+\infty}\int_EC_t(e)p(dtde)}=\evals{\int_0^{+\infty}\int_EC_t(e)\nu(dtde)}.
$$
In particular, this measure can be decomposed as $\nu(\omega,dtde)=\phi_t(\omega,de)dA_t(\omega)$. Moreover the following properties hold:
\begin{itemize}
	\item for every $\omega\in\Omega$, $t\in[0,+\infty)$,  $C\mapsto\phi_t(\omega,C)$ is a probability on $(E,\mathcal{E})$.
	\item for every $C\in\mathcal{E}$, the process $\phi_t(C)$ is predictable.
\end{itemize}
In the following that all marked point processes in this paper have a compensator of this form. Here we have one of the main assumptions
\begin{Asmpt}{(A)}
	\label{ass:continuous_A}
	The process $A$ is continuous.
\end{Asmpt}
It is equivalent (see \cite[Corollary 5.28]{sheng1998semimartingale}) to stating that the counting process $N_t=p((0,t],E)$ is quasi-left continuous, i.e. it has only totally inaccessible jumps.

From now on fix a terminal time $T>0$. Define now the compensated measure $q(dtde)=p(dtde)-\phi_t(de)dA_t$. \begin{defi}
	Let $C$ be a $\mathcal{P}^{\mathcal{G}}\otimes\mathcal{E}$-measurable process such that
	$$
	\int_0^T\int_E |C_t(e)|\phi_t(de)dA_t<\infty.
	$$
	Then we can define the integral
	$$
	\int_0^T\int_E C_t(e)q(dtde)=\int_0^T\int_EC_t(e)p(dtde)-\int_0^T\int_EC_t(e)\phi_t(de)dA_t
	$$
	as difference of ordinary integrals with respect to $p$ and $\phi dA$.
\end{defi}
\begin{rem}
	In the paper we adopt the convention that $\int_a^b$ denotes an integral on $(a,b]$ if $b<\infty$, or on $(a,b)$ if $b=\infty$.
\end{rem}
\begin{rem}
	Since $p$ is a discrete random measure, the integral with respect to $p$ is a sum:
	$$
	\int_0^t\int_E C_s(e)p(dsde)=\sum_{T_n\leq t}C_{T_n}(\xi_n)
	$$
\end{rem}

We have the following result

\begin{prop}
	Let $C$ be a predictable random field that satisfies
	$$
	\eval\left[\int_0^T\int_E |C_t(e)|\phi_t(de)dA_t\right]<\infty.
	$$
	Then the integral
	\begin{equation}
	\label{eq:generic_MPP_martingale}
	\int_0^t\int_E C_s(e)q(dsde)
	\end{equation}
	is a martingale.
\end{prop}
In the following $\mathcal{E}(H)$ denotes the Doleans-Dade exponential of the process $\int_0^\cdot\int_E (H_s(e)-1)q(dsde)$, that is
$$
\mathcal{E}(H)_t=\prod_{0<T_n\leq T}\rho_{T_n}(\xi_n)e^{\int_0^t\int_E (1-\rho_s(e))\phi_s(de)dA_s}.
$$
For a comprehensive treatment of these type of processes,
we refer the reader to \cite{jacod1975multivariate}, \cite{bremaud1981point}
or \cite{lastbrandt1995}.

For some parameter $\beta>0$, we introduce the following spaces:
\begin{itemize}
	\item $\mathcal{L}^{r,\beta}(A)$  is the space of all
	$\mathbb{F}$-progressive processes $X$ such that
	$$||X||_{\mathcal{L}^{r,\beta}(A)}^r=\eval\left[\int_0^Te^{\beta A_s}|X_s|^rdA_s\right]<\infty.$$
	\item $\mathcal{L}^{r,\beta}(p)$ 
	is the space of all $\mathbb{F}$-predictable 
	processes $U$ such that $$||U||_{\mathcal{L}^{r,\beta}(p)}^r
	=\eval\left[\int_0^T\int_E e^{\beta A_s}|U_s(e)|^r\phi_s(de)dA_s\right]<\infty.$$
	\item $\mathcal{L}^{r,\beta}(W,\mathbb{R}^d)$ 
	is the space of $\mathbb{F}$-progressive 
	processes $Z$ in $\mathbb{R}^d$ such that $$||Z||_{\mathcal{L}^{r,\beta}(W)}^r=\eval\left[\int_0^Te^{\beta A_s}|Z_s|^rds\right]<\infty$$
	\item ${I}^2$ ({resp.} ${I}^2(\mathbb{G})$)
	is the space of all càdlàg increasing $\mathbb{F}$-predictable  processes $K$ such that $\eval[K^2_T]<\infty$.
\end{itemize}
We indicate by ${L}^{r,\beta}(A)$, ${L}^{r,\beta}(p)$, ${L}^{r,\beta}(W)$ and $\mathcal{I}^2$ the respective space of equivalence classes. Notice that $L^{2,\beta}(p)\subset L^{1,0}(p)$ thanks to H\"older inequality.

Let us now describe the switching problem. Let $\mathcal{J}=\lbrace1,\dots,m\rbrace$. For each mode $i\in \mathcal{J}$, a terminal reward $\xi^i$ is given, as well as running gains $f_s^i$ and $g_s^i$. Moreover, for each $i\in\mathcal{J}$, we consider a non-negative process $\rho^i_s(e)$ that is $\mathcal{P}\otimes \mathcal{E}$ measurable. The cost of switching from mode $i$ to mode $j$ at time $t$ is given by $C_t(i,j)$, a non negative process.

As always, the controller can switch between modes by choosing switching times and actions. In particular, we define for each mode $i$ the sets $\mathbb{A}_i$ of possible switching destinations as $\mathbb{A}_i=\mathcal{J}\setminus {i}$. A strategy is a sequence of couples $(\theta_n,\alpha_n)$ where $\theta_n$ is a stopping time and $\alpha_n$ is a $\mathcal{J}$-valued $\mathcal{F}_{\theta_n}$-adapted random variable. The law of the point process depends on the switching status, as the switching mode sets a different compensator for the point process through the functions $\rho^i$. We assume the following on the $\rho^i$ and $A$

\begin{Asmpt}{(S)}
	\label{ass:rhosbound}
	The $\rho^i$ satisfy $0\leq \rho^i_t(e) \leq M$ and there exists $\eta>3+M^4$ such that $\evals{e^{\eta A_T}}<\infty$.
\end{Asmpt}
Define $M'=\max(|M-1|,1)$. In the following we will often use an absolutely continuous change of probability ``à la Girsanov", where the ``Girsanov kernel" will be one of the $\rho^i$ or a combination of them, as described in the following. 
Let us first define an admissible switching strategy and the corresponding switched kernel.

\begin{defi}
	A strategy is called admissible if $P(\theta_n<T \;\forall n)=0$. The set of admissible strategies such that $(\theta_0,\alpha_0)=(t_i)$ is denoted as $\mathcal{A}_t^i$.
\end{defi}

For a strategy $(\theta_0,\alpha_0)$ we define the process $a$ indicating the current mode as

\begin{equation}
\label{eq:strategy_defined}
a_s=\sum_{k\geq 1}\alpha_{k-1}\ind_{(\theta_{k-1},\theta_k]}(s)
\end{equation}

When the controller changes mode, the law of the process changes in the following way. For a strategy $\mathbf{a}\in\mathcal{A}_t^i$, consider the following $\mathcal{P}\otimes\mathcal{E}$-measurable process process
\begin{equation}
\rho_s^{\mathbf{a}}=\sum_{k\geq 1}\ind_{(\theta_{k-1},\theta_{k}]}(s)\rho_s^{\alpha_{k-1}}(e)
\end{equation}

Now introduce the supermartingale $L_t^\mathbf{a}=\mathcal{E}(\rho^\mathbf{a})$:
\begin{equation}
\label{eq:Zdefined}
L_t^\mathbf{a}=\prod_{k\geq 1}^{}\rho^\mathbf{a}_{T_n}(\xi_n)e^{\int_0^t\int_E(1-\rho^{\mathbf{a}}_s(e))\phi_s(de)dA_s}
\end{equation}
When $L^\mathbf{a}$ is a  martingale, we define the absolutely continuous probability $\mathbb{P}^\mathbf{a}\ll\mathbb{P}$ as
$$
\frac{\mathbb{P}^\mathbf{a}}{\mathbb{P}}=L_T^\mathbf{a}.
$$
In this case, under $\mathbb{P}^\mathbf{a}$ the compensator of $p$ becomes
\begin{equation}
\nu^{\mathbf{a}}(dsde)=\rho_s^{\mathbf{a}}\phi_s(de)dA_s.
\end{equation}
Thanks to assumptions above, we have that $L_t^\mathbf{a}$ is a martingale for any choice of $\mathbf{a}\in\mathcal{A}_t^i$:
\begin{prop}
	\label{prop:still_mg_and_square}
	Let $\mathbf{a}\in\mathcal{A}_t^i$ be fixed. Under assumption \ref{ass:rhosbound}, $L^\mathbf{a}$ is a $\mathbb{P}$-martingale with $\sup_{t\in[0,T]}\eval\left[(L^\mathbf{a}_t)^2\right]<+\infty$. Moreover it holds that
	\begin{equation}
	\label{eq:square_martingale_conditioned}
	\frac{\eval\left[(L^\mathbf{a}_t)^2\right.\left|\vphantom{(L^a)^2}\mathcal{F}_s\right]}{(L^\mathbf{a}_s)^2}< \econd{e^{\eta A_T}}\quad\mathbb{P}\text{-a.s.}
	\end{equation}
	If $H$ is a process in $L^{1,0}(p)\cap L^{2,0}(p)$, then $\int_0^t\int_EH_s(e)q^\mathbf{a}(dsde)$ is a $\mathbb{P}^\mathbf{a}$-martingale, where $q^\mathbf{a}(dsde)=p(dsde)-\rho^\mathbf{a}_s(e)\phi_s(de)dA_s$. Also if $M$ is a $\prob$-martingale of the form
	$$
	M_t=M_0+\int_0^tZ_sdW_s
	$$
	for some process $Z$ in $L^{2,0}(W)$, then $M$ is also a $\prob^\mathbf{a}$-martingale.
\end{prop}
\begin{proof}
	From the definition of $\rho^\mathbf{a}$ we have that $\rho^\mathbf{a}$ verifies assumption \ref{ass:rhosbound} too, that is $0 \leq \rho^\mathbf{a}_t(e)\leq M$. 
	The proof of the first part can be found in \cite[Lemma 4.2]{confortola2013backward}. The authors prove that if for some $\gamma>1$ and
	$$
	\beta=\gamma+1+\frac{M^{\gamma^2}}{\gamma-1}
	$$
	it holds that $\evals{e^{\beta A_T}}<\infty$, then $\sup_{t\in[0,T]}\evals{(L_t^\mathbf{a})^\gamma}<\infty$ and $\evals{L^\mathbf{a}_T}=1$. We don't repeat the proof here, but mention that the proof still holds if we choose $\beta$ to be any number strictly greater than $\gamma+1+\frac{M^{\gamma^2}}{\gamma-1}$. In particular the statement of our theorem refers to the case $\gamma=2$.
	As for the proof of the inequality \eqref{eq:square_martingale_conditioned} we have (which is a generalization of what done for the first part of the lemma).
	The quantity $$\frac{\eval\left[(L^\mathbf{a}_t)^2\right.\left|\vphantom{(L^a)^2}\mathcal{F}_s\right]}{(L^\mathbf{a}_s)^2}$$
	can be rewritten as
	\begin{multline*}
	\eval\left[
	\frac{\prod\limits_{0<T_n\leq t}\rho^\mathbf{a}_{T_n}(\xi_n)}{\prod\limits_{0<T_n\leq s}\rho^\mathbf{a}_{T_n}(\xi_n)}\exp\left\lbrace\frac{1}{2}\int_s^t\int_E\left(1-(\rho^\mathbf{a}_s(e))^4\right)\phi_s(de)dA_s\right\rbrace\cdot\right.\\
	\exp\left\lbrace 2\int_s^t\int_E\left(1-\rho^\mathbf{a}_s(e)\right)\phi_s(de)dA_s-\frac{1}{2}\int_s^t\int_E\left(1-(\rho^\mathbf{a}_s(e))^4\right)\phi_s(de)dA_s\right\rbrace
	\left|\vphantom{\frac{\prod\limits_{0<T_n\leq T}\rho^\mathbf{a}_{T_n}(\xi_n)}{\prod\limits_{0<T_n\leq T}\rho^\mathbf{a}_{T_n}(\xi_n)}}\mathcal{F}_s\right].
	\end{multline*}
	Using Cauchy-Scwarts inequality for conditional expectations, this becomes
	\begin{multline*}
	\eval\left[\exp\left\lbrace 4\int_s^t\int_E(1-\rho^\mathbf{a}_s(e))\phi_s(de)dA_s-\int_s^t\int_E\left(1-(\rho^\mathbf{a}_s(e))^4\right)\phi_s(de)dA_s\right\rbrace\right.\left|\vphantom{\left\lbrace 4\int_s^t\int_E(1-\rho^\mathbf{a}_s(e))\phi_s(de)dA_s\right\rbrace}\mathcal{F}_s\right]^{1/2}\\
	\eval\left[\frac{\mathcal{E}((\rho^\mathbf{a})^4)_t}{\mathcal{E}((\rho^\mathbf{a})^4)_s}\right.\left|\vphantom{\frac{\mathcal{E}((\rho^\mathbf{a})^4)_t}{\mathcal{E}({\rho^\cdot\mathbf{a}}^4)_s}}\mathcal{F}_s\right]^{1/2}.
	\end{multline*}
	The last term is less or equal than $1$, thanks to the supermartingale property of $\mathcal{E}({\rho}^4)$, and the other term is easily estimated:
	\begin{multline*}
	\econd[\mathcal{F}_s]{\exp\left\lbrace \int_s^t\int_E(3-4\rho_s(e))\phi_s(de)dA_s+\int_s^t\int_E(\rho_s(e))^4\phi_s(de)dA_s\right\rbrace}\\
	\leq \econd[\mathcal{F}_s]{\exp\left\lbrace (3+L^4)(A_t-A_s)\right\rbrace}< \econd[\mathcal{F}_s]{e^{ \eta A_T}}.
	\end{multline*}
	We now turn to the last claim. 
	Clearly $\int_0^tZ_sdW_s$ is a local $\prob^\mathbf{a}$-martingale.
	By using Burkholder-Davis-Gundy inequality (under $\prob^\mathbf{a}$) we have that
	\begin{equation*}
	\eval^\mathbf{a}\left[\sup_{t\in[0,T]}\left|\int_0^tZ_sdW_s\right|\right]
	\leq K \evals{(L^\mathbf{a}_T)^2}^{\frac{1}{2}} \evals{\int_0^T|Z_s|^2ds}^\frac{1}{2}
	<\infty.
	\end{equation*}
	Thus the integral is a $\prob^\mathbf{a}$-martingale.
As for the point process part we do the same obtaining
\begin{align*}
\eval^\mathbf{a}\left[\sup_{t\in[0,T]}\left|\int_0^t\int_E U_s(e)q^\mathbf{a}(dsde)\right|\right]
&\leq K \eval^\mathbf{a}\left[\left(\int_0^T\int_E |U_s|^2(e)p(dsde)\right)^{1/2}\right]\\
&= K \evals{L^\mathbf{a}_T\left(\int_0^T\int_E|U_s(e)|^2p(dsde\right)^{1/2}}\\
&\leq K  \evals{(L^\mathbf{a}_T)^2}^{\frac{1}{2}} \evals{\int_0^T\int_E|U_s(e)|^2\phi_s(de)dA_s}^\frac{1}{2}\\
&<\infty.
\end{align*}
\end{proof}

We can finally formulate the optimal switching problem. For a given strategy $\mathbf{a}\in\mathcal{A}_t^i$, introduce the target quantity to maximise
\begin{equation}
J(t,i,\mathbf{a})=\eval^{\mathbf{a}}\left[\left.\xi^{a_T}+\int_t^Tf_s^{a_s}dA_s+\int_t^Tg_s^{a_s}ds-\sum_{k\geq 1}C_{\theta_k}(\alpha_{k-1},\alpha_k)\right|\mathcal{F}_t\right].
\end{equation}
We then have the value function 
\begin{equation}
v(t,i)=\essup_{\mathbf{a}\in\mathcal{A}_t^i}J(t,i,\mathbf{a}).
\end{equation}
The aim is to find this optimal value and characterize the optimal strategy.
\begin{rem}
	We have chosen to let the kernels $\rho^i$ to touch the value zero, which is useful if we want to have one of the modes to completely cut off the dynamic for a limited time. This comes with added technical difficulties, since the probabilities $\prob^i$ and $\prob^\mathbf{a}$ are not equivalent to the reference probability $\prob$. In particular relations that can be easily shown to hold $\prob^i$-a.s. or $\prob^\mathbf{a}$-a.s. need not to hold $\prob$-a.s. In order to obtain some results we will make use of approximated kernels that induce an equivalent probability.
\end{rem}

In order to tackle the problem, we will represent the value function of each mode through the use of a system of reflected BSDE. We introduce the following system of Backward equations:

\begin{equation}
\label{eq:system}
\begin{cases}
Y_t^i=\xi^i+\int_t^Tf_s^idA_s+\int_t^T g^i_sds+\int_t^T\int_E U_s^i(e)(\rho_s^i(e)-1)dA_s\\
\qquad -\int_t^TU_s^i(e)q(dsde)-\int_t^TZ_s^idW_s+K_T^i-K_t^i\\
Y_t^i\geq \max\limits_{j\in \mathbb{A}_i}(Y^j_t-C_t(i,j))\\
\int_0^T(Y_t^i- \max\limits_{j\in \mathbb{A}_i}(Y^j_t-C_t(i,j)))dK_t^i=0.
\end{cases}
\end{equation}

In \eqref{eq:system} we see that the generator of the $i$-th equation only depends on $U^i$, and it does so in a very specific way. This will allow us to consider the system not only under the reference probability $\prob$, but also under the probability $\prob^i$ induced by the kernel $\rho^i$. In this second case the whole term 
$$
+\int_t^T\int_EU_s^i(e)(\rho_s^i(e)-1)dA_s-\int_t^TU_s^i(e)q(dsde)
$$
will be a $\prob^i$ martingale. This means we can incorporate in the system of BSDE the fact that in the switching problem the compensator in mode $i$ changes to $\rho_s^i(e)\phi_s(de)dA_s$.
A solution to the BSDE is $(Y^i,U^i,Z^i,K^i)_{i\in \mathcal{J}}$ where, for all $i\in\mathcal{J}$ and for all $(M')^2<\beta<\hat{\beta}$,
\begin{itemize}
	\item $Y_i$ is an adapted càdlàg process in $L^{2,\beta}(A)\cap L^{2,\beta}(W)$.
	\item $U_i$ is in $L^{2,\beta}(p)$.
	\item $Z_i$ is in $L^{2,\beta}(W)$.
	\item $K^i_t$ is an increasing continuous process in $\mathcal{I}^2$.
\end{itemize}

First we start with some assumptions, that will serve both to assure that the switching problem is well posed and that there exists a solution to the system of BSDE:
\begin{Asmpt}{(D)}
	\label{ass:data_swit}
	For $i\in\mathcal{J}$ and for some $\hat{\beta}>(M')^2$
	\begin{enumerate}[label=\roman*)]
		\item \label{swit_assdata_final} $\xi^i$ is a given $\mathcal{F}_T$ measurable variable such that $\eval^i[e^{\hat{\beta}A_T}|\xi^i|^2]<\infty$.
		\item \label{swit_assdata_f} $f_s^i$ is a progressive process in $L^{2,\hat{\beta}}(A)$.
		\item \label{swit_assdata_g} $g_s^i$ is a progressive process in $L^{2,\hat{B}}(W)$.
		\item \label{swit_assdata_costs} $C_t(i,j)$ for $i,j\in\mathcal{J}$ are continuous adapted processes. $C_t(i,j)\geq 0$ for $i\neq j$ and $C_t(i,i)=0$ and \begin{equation}
		\label{ass:no_free_cycles}\inf_t(C_t(i,j)+C_t(j,l)-C_t(i,l))>0\text{ for all }i\neq j\neq l.\end{equation}
		Moreover, for all $i,j\in\mathcal{J}$, $\eval\left[\sup_te^{\hat{\beta}A_t}C_t(i,j)\right]<\infty$.
		\item \label{swit_assdata_final_greater} It holds that for all $i,j\in\mathcal{J}$ $\xi^i\geq\xi^j-C_T(i,j)$ a.s.
	\end{enumerate}
\end{Asmpt}
Assumptions \ref{swit_assdata_final}, \ref{swit_assdata_f}, \ref{swit_assdata_g} and \ref{swit_assdata_final_greater}  are the usual ones needed for the reflected BSDE, while  \ref{swit_assdata_costs} is typical of switching problems. In particular \eqref{ass:no_free_cycles} means that it is not possible to switch from $i$ to $j$ by switching to $k$ in the middle for a lower cost.
\begin{rem}
	The condition \eqref{ass:no_free_cycles} is a bit stronger than the other usual condition taken in switching problems, that is the no free loop property. This property states that for any cycle of indexes of any length $p$ $j_1,j_2,\dots,j_p=j_1$ it must hold that
	$$
	\sum_{k=1}^{p-1}C_t(j_k,j_{k+1})>0.
	$$
	Details of how \eqref{ass:no_free_cycles} implies the no free loop property can be found in the proof of proposition \ref{swit:Kconti}.
\end{rem}

\begin{rem}
	\label{swit_another_no_free_loop}
	We have formulated the problem in the case of a marked point process and a Brownian motion, but it is of course possible to consider the case where only a marked point process is involved. In that case it is enough to put $g\equiv 0$ and consider all data adapted to the filtration generated by the point process. The solution would then be just $(Y^i,U^i,K^i)_{i\in \mathcal{J}}$. In some sense, we could say that the Wiener process does not modify the nature of the current problem. This is because the change of probability for the point process does not in any way modify the behaviour of the diffusive part.
\end{rem}

\section{Existence of a solution to the system of RBSDE}
\label{swit_existence}
In this section we obtain existence of a solution to the system \eqref{eq:system}. To this end we construct the solution iteratively, where at step $n$, the barrier is represented by the solution at step $n-1$. This technique was used for example in \cite{hamadene2010switching}, \cite{hamadene2015systems} and \cite{chassagneux2011note}. One of the most important tools here is the comparison theorem, as we have to show that our iterative scheme converges increasingly to some point and that this point is the solution we look for. The comparison theorem cannot be applied in all cases (for example when we compare a reflected BSDE with a one without minimal push condition), thus in some parts we resort to direct comparisons. What we will obtain is

\begin{teo}
	Under assumptions \ref{ass:rhosbound} and \ref{ass:data_swit}, there exists a solution to the system \eqref{eq:system} such that for all $i$ $(Y^i,U^i,Z^i,K^i)\in L^{2,\beta}(A)\cap L^{2,\beta}(W)\times L^{2,\beta}(p) \times L^{2,\beta}(W) \times \mathcal{I}^2$ for $(M')^2<\beta<\hat{\beta}$.
\end{teo}

From now on, let assumptions \ref{ass:continuous_A}, \ref{ass:data_swit} and \ref{ass:rhosbound} hold. Fix a $\beta$ such that $(M')^2<\beta=\hat{\beta}-\delta$ for some fixed $\delta>0$. First consider the function defined on $[0,T]\times\mathbb{R}\times E$ as $$h_s(u,e)=u\max_{i\in \mathcal{J}}(\rho_s(e)^i-1)\indb{u\geq 0}+u\min_{i\in \mathcal{J}}(\rho_s(e)^i-1)\indb{u< 0}.$$

Notice $\int_Eh_s(U(e))\phi_s(de)$ satisfies the Lipschitz condition:
\begin{multline*}
|\int_Eh_s(U'(e))-h_s(U(e))\phi_s(de)|\leq\int_EM'|U'(e)-U(e)|\phi_s(de)\\\leq M'\left(\int_E|U'(e)-U(e)|^2\phi_s(de)\right)^{1/2}.
\end{multline*}
Define for $i\in\mathcal{J}$ $(Y^{i,0},U^{i,0},Z^{i,0})$ as solution to
\begin{multline}
\label{eq:startingY0}
Y^{i,0}_t=\xi^i+\int_t^Tf_s^idA_s+\int_t^Tg_s^ids+\int_t^T\int_EU_s^{i,0}(e)(\rho_s^i(e)-1)\phi_s(de)dA_s\\-\int_t^T\int_E U^{i,0}(e)q(dsde)-\int_t^TZ_s^{i,0}dW_s,
\end{multline}
and $(\hat{Y},\hat{U},\hat{Z})$ solution to
\begin{multline}
\label{eq:upper_limit_Y}
\hat{Y}_t=\max_i|\xi^i|+\int_t^T\max_i|f_s^i|dA_s+\int_t^T\max_i|g_s^i|ds\\+\int_t^T\int_Eh_s(\hat{U}_s(e),e)\phi_s(de)dA_s-\int_t^T\int_E\hat{U}(e)q(dsde)-\int_t^T\hat{Z}_sdW_s.
\end{multline}

\begin{prop}
	\label{swit:Y0andYhat}
	Let assumptions \ref{ass:rhosbound} and \ref{ass:data_swit} hold. For all $i\in\mathcal{J}$, there exist unique solutions $(Y^{i,0},U^{i,0},Z^{i,0})$ and $(\hat{Y},\hat{U},\hat{Z})$ in $L^{2,\hat{\beta}}(A)\cap L^{2,\hat{\beta}}(W)\times L^{2,\hat{\beta}}(p)\times L^{2,\hat{\beta}}(W)$ to \eqref{eq:startingY0} and \eqref{eq:upper_limit_Y} respectively. Moreover it holds that $Y^{i,0}_t\leq \hat{Y}_t$ $\prob$-a.s. and that
	\begin{equation}
	\label{eq:sup_bounded}
	\eval\left[\sup\limits_{t\in [0,T]}e^{\hat{\beta} A_t}\left(|Y_t^{i,0}|^2+|\hat{Y}_t|^2\right)\right]<\infty.
	\end{equation}
\end{prop}

\begin{proof}
	Existence and uniqueness in $L^{2,\hat{\beta}}(A)\cap L^{2,\hat{\beta}}(W)\times L^{2,\hat{\beta}}(p)\times L^{2,\hat{\beta}}(W)$ are provided by theorem \ref{appendix_simple_BSDE}
	
	To prove that $Y^{i,0}_t\leq\hat{Y}_t$ we can use the comparison result \ref{appendix_BSDE_teo_general_comparison_simple_BSDE}. Indeed it is sufficient to notice that 
	$$f_s^i+\int_E u_s(e)(\rho_s^i(e)-1)\phi_s(de)\leq \max_i|f_s^i|+\int_E h_s(u_s(e),e)\phi_s(de). $$
	This is true thanks to the definition of $h$. Indeed if we explicit it we have
	\begin{multline*}
	u_s(e)(\rho_s^i(e)-1)-u_s(e)\max_{i\in \mathcal{J}}(\rho_s(e)^i-1)\indb{u_s(e)\geq 0}\\-u_s(e)\min_{i\in \mathcal{J}}(\rho_s(e)^i-1)\indb{u_s(e)< 0}\leq 0.
	\end{multline*}
	As for the last point of the proposition, it can be shown directly. Let us see it for $\hat{Y}$, since for $Y^{i,0}$ is analogous. Take expected value conditioned on $\mathcal{F}_t$ on the equation for $\hat{Y}$, then consider the absolute value:
	\begin{align}
	\label{switching_another_sup_estimate}
	\begin{split}
	e^{\hat{\beta}A_t/2}|\hat{Y}_t|&\leq\econd{e^{\hat{\beta}A_t/2}|\hat{\xi}|+e^{\hat{\beta}A_t/2}\int_t^T|\hat{f}_s|dA_s+e^{\hat{\beta}A_t/2}\int_t^T|\hat{g}_s|ds}\\
	&+\econd{+e^{\hat{\beta}A_t/2}\int_t^T\int_E|h_s(\hat{U}_s)(e)|\phi_s(de)dA_s}\\
	&\leq \econd{e^{\hat{\beta}A_T/2}|\hat{\xi}|+\frac{1}{\hat{\beta}^{1/2}}\left(\int_0^Te^{\hat{\beta}A_s}|\hat{f}_s|^2dA_s\right)^{1/2}}\\
	&+\econd{\sqrt{T}\left(\int_0^Te^{\hat{\beta}A_s}|\hat{g}_s|^2ds\right)^{1/2}}\\
	&+\frac{M'}{\hat{\beta}^{1/2}}\econd{\left(\int_0^T\int_Ee^{\hat{\beta}A_s}|\hat{U}_s(e)|^2\phi_s(de)dA_s\right)^{1/2}}.
	\end{split}
	\end{align}
	The last inequality holds thanks to the following
	\begin{align*}
	e^{\hat{\beta}A_t/2}&\int_t^T\int_E|h(\hat{U}_s(e))|\phi_s(de)dA_s\\
	&\leq M'e^{\hat{\beta}A_t/2}\int_t^T\int_E|\hat{U}_s(e)|\phi_s(de)dA_s\\
	&=M'e^{\hat{\beta}A_t/2}\int_t^T\int_Ee^{-\hat{\beta}A_s/2}e^{\hat{\beta}A_s/2}|\hat{U}_s(e)|\phi_s(de)dA_s\\
	&\leq M'e^{\hat{\beta}A_t/2}\left(\int_t^Te^{-\hat{\beta}A_s}dA_s\right)^{1/2}
	\left(\int_0^T\int_Ee^{\hat{\beta}A_s}|\hat{U}_s(e)|^2\phi_s(de)dA_s\right)^{1/2}\\
	&=M'e^{\hat{\beta}A_t/2}\left(\frac{e^{-\hat{\beta}A_t}-e^{-\hat{\beta}A_T}}{\hat{\beta}^{1/2}}\right)^{1/2}
	\left(\int_0^T\int_Ee^{\hat{\beta}A_s}|\hat{U}_s(e)|^2\phi_s(de)dA_s\right)^{1/2}\\
	&\leq\frac{M'}{\hat{\beta}^{1/2}}\left(\int_0^T\int_Ee^{\hat{\beta}A_s}|\hat{U}_s(e)|^2\phi_s(de)dA_s\right)^{1/2},
	\end{align*}
	and analogously for the part with $\hat{f}$. As for the integral of $g$ we have
	$$
	e^{\hat{\beta}A_t/2}\int_t^T|\hat{g}_s|ds\leq\int_t^Te^{\hat{\beta} A_s/2}|\hat{g}_s|ds\leq \sqrt{T}\left(\int_0^Te^{\hat{\beta}A_s}|\hat{g}_s|^2dA_s\right)^{1/2}.
	$$
	The term on the right of inequality \eqref{switching_another_sup_estimate} is a martingale. By Doob's inequality, the expected value of its $\sup$ is bounded by the quantity
	\begin{multline*}
	C\evals{e^{\hat{\beta}A_T}|\hat{\xi}|^2+\frac{1}{\hat{\beta}}\int_0^Te^{\hat{\beta}A_s}|\hat{f}_s|^2dA_s+T\int_0^Te^{\hat{\beta}A_s}|\hat{g}_s|^2ds}\\
	+\frac{(M')^2}{\hat{\beta}}\evals{\int_0^T\int_Ee^{\hat{\beta}A_s}|\hat{U}_s(e)|^2\phi_s(de)dA_s},
	\end{multline*}
	which is finite.
\end{proof}

Notice that $(\hat{Y},\hat{U},\hat{Z},\hat{K})$, with $\hat{K}\equiv0$, for all $i$ solves the Reflected BSDE
\begin{equation}
\label{eq:limiting_RBSDE}
\begin{cases}
\hat{Y}_t=\hat{\xi}+\int_t^T\max_i|f_s^i|dA_s+\int_t^T \max_i|g_s^i|ds+\int_t^T\int_Eh_s(\hat{U}_s(e))\phi_s(de)dA_s\\
\qquad-\int_t^T\hat{U}(e)q(dsde)-\int_t^T\hat{Z}_sdW_s+\hat{K}_T-\hat{K}_t\\
\hat{Y}_t\geq\max\limits_{j}(\hat{Y}_t-C_t(i,j))\\
\int_0^T(\hat{Y}_t-\max\limits_{j}(\hat{Y}_t-C_t(i,j)))d\hat{K}_t=0
\end{cases}
\end{equation}

Consider now the sequence of RBSDEs:

\begin{equation}
\label{eq:sequence_RBSDE}
\begin{cases}
Y^{i,n}_t=\xi^i+\int_t^Tf_s^idA_s+\int_t^Tg_s^ids+\int_t^T\int_EU^{i,n}_s(e)(\rho_s^i(e)-1)\phi_s(de)dA_s)\\
\qquad -\int_t^TU_s^{i,n}(e)q(dsde-\int_t^TZ_s^idW_s+K^{i,n}_T-K^{i,n}_t\\
Y^{i,n}_t\geq\max\limits_{j\in \mathbb{A}_i}(Y^{i,n-1}_t-C_t(i,j))\\
\int_0^T(Y^{i,n}_{t^-}-\max\limits_{j\in \mathbb{A}_i}(Y^{i,n-1}_{t^-}-C_t(i,j)))dK^{i,n}_t=0
\end{cases}
\end{equation}

\begin{prop}
	\label{swit_existence_picard_iteration}
	For all $i\in\mathcal{J}$ and for all $n\geq 1$, \eqref{eq:sequence_RBSDE} admits a unique solution in $L^{2,\beta}(A)\cap L^{2,\beta}(W)\times L^{2,\beta}(p) \times  L^{2,\beta}(W)$. Moreover it holds that $Y^{i,0}_t\leq Y^{i,n}_t\leq Y^{i,n+1}_t\leq \hat{Y}_t$.
\end{prop}

\begin{proof}
	We have the starting point $Y^{i,0}$, and existence and uniqueness of $Y^{i,1}$ is provided by  \cite[Theorem 4.1]{foresta2017optimal}. For every $n$, we can construct the solution that has $\max_{i\in \mathbb{A}_i}(Y^{j,n-1}-C_t(i,j))$ as known barrier.
	First we show by direct verification  that for all $i$, $Y^{i,1}\geq Y^{i,0}$.
	By considering the difference $\bar{Y}=Y^{i,1}-Y^{i,0}$ we obtain, noting that $K_T^{i,1}-K_t^{i,0}\geq 0$,

	\begin{equation}
	\bar{Y}_t\geq\int_t^T\int_E\bar{U}_s(e)(\rho_s(e)-1)\phi_s(de)dA_s-\int_t^T\int_E \bar{U}_s(e)q(dsde)-\int_t^T\bar{Z}_sdW_s.
	\label{swit_first_minus_zero}
	\end{equation}
	The right hand term in \eqref{swit_first_minus_zero} is a $\prob^{\rho^i}$-martingale, but since $\prob^{\rho^i}$ is not equivalent to $\prob$, we cannot conclude that $\bar{Y}_t\geq 0$ by taking expectation. Instead
	 we add to both sides of the equation the quantity
	$$
	\epsilon\int_t^T\int_E\bar{U}_s(e)\phi_s(de)dA_s,
	$$
	for any $\epsilon<\bar{\epsilon}$, where $\bar{\epsilon}$ is such that $3+(M+\bar {\epsilon})^4<\eta$. 
	We obtain 
	\begin{multline}
	\bar{Y}_t+\epsilon\int_t^T\int_E\bar{U}_s(e)\phi_s(de)dA_s\geq-\int_t^T\int_E \bar{U}_s(e)(p(dsde)-(\rho_s(e)+\epsilon)\phi_s(de)dA_s)\\-\int_t^T\bar{Z}_sdW_s.
	\label{swit_first_minus_zero_changed}
	\end{multline}
	$\rho^i+\epsilon$ satisfies assumptions \ref{ass:rhosbound} and it induces a probability $\prob^{\rho^i+\epsilon}\sim\prob$ through the ``Girsanov" martingale
	$$
	L_t^{\rho^i+\epsilon}=\mathcal{E}(\rho^i+\epsilon).
	$$
The term in the right hand side of \eqref{swit_first_minus_zero_changed} is a $\prob^{\rho^i+\epsilon}$-martingale we can take expected value under $\prob^{\rho+i}$ conditional on $\mathcal{F}_t$ and obtain that for all $\epsilon <\bar{\epsilon}$
	$$
	\bar{Y}_t+\epsilon\econdm{\int_t^T\int_E\bar{U}_s(e)\phi_s(de)dA_s}{\rho+\epsilon}\geq 0 \qquad \prob^{\rho^i+\epsilon}\text{-a.s.}
	$$
	and thus also $\prob$-a.s.
	Next we need to show that for $\epsilon\rightarrow 0$, the term
	$$
	\epsilon\econdm{\int_t^T\int_E\bar{U}_s(e)\phi_s(de)dA_s}{\rho+\epsilon}\rightarrow 0.
	$$
	We only need to show that $\econdm{\int_t^T\bar{U}_s(e)\phi_s(de)dA_s}{\rho+\epsilon}\rightarrow 0$ is bounded by a finite random variable. We have
	\begin{align*}
	&\econdm{\int_t^T\bar{U}_s(e)\phi_s(de)dA_s}{\rho+\epsilon}\leq\econdm{\int_t^T\int_E|\bar{U}_s(e)|\phi_s(de)dA_s}{\rho+\epsilon}\\
	&=\frac{\econd{L_T^{\rho^i+\epsilon}\int_t^T\int_E|\bar{U}_s(e)|\phi_s(de)dA_s}}{L_t^{\rho^i+\epsilon}}\\
	&\leq \econd{\left(\frac{L_T^{\rho^i+\epsilon}}{L_t^{\rho^i+\epsilon}}\right)}\cdot\econd{\left(\int_t^T\int_E|\bar{U}_s(e)|\phi_s(de)dA_s\right)^2}\\
	&\leq \frac{1}{\beta}\econd{e^{\eta A_T}}\cdot\econd{\int_t^Te^{\beta A_s}|\bar{U}_s|^2(e)\phi_s(de)dA_s},
	\end{align*}
	where in the last line we used relation \eqref{eq:square_martingale_conditioned} and the usual trick (see the proof of proposition \ref{swit:Y0andYhat}) to estimate the square of an integral with respect to the compensator.
	
	Using the comparison theorem \ref{otherBSDE_comparison}, we also obtain that $\hat{Y}\geq{Y}^{i,1}$ for all $i$.
	Indeed, the obstacle $\max_j(Y^{j,0}-C_t(i,j))$ for $Y^{i,1}$ is smaller than the one for $\hat{Y}$, and we can transform the equation for $\hat{Y}$ into the suitable form by adding and subtracting $\int_t^T\int_E \hat{U}_s(e)(\rho_s^i(e)-1)\phi_s(de)dA_s$. This way the generator of the equation for $\hat{Y}$ becomes
	$$
	\max_i|f_s^i|+\int_E \underbracket{\left[h_s(\hat{U}_s(e),e)-\hat{U}_s(e)(\rho_s^i(e)-1)\right]}_{\ge 0}\phi_s(de) +\int_E\hat{U}_s(e)(\rho_s^i(e)-1)\phi_s(de)
	$$
	and thus 
	$$
	f_s^i\leq \max_i|f_s^i|+\int_E h_s(\hat{U}_s(e),e)-\hat{U}_s(e)(\rho_s^i(e)-1)\phi_s(de),
	$$
	and we can apply the theorem.
	By induction we can then prove that $Y^{i,n}\leq Y^{i,n+1}\leq \hat{Y}$, repeating the same reasoning.
\end{proof}

This means that we have an increasing sequence of processes $Y^{i,n}$, and then there exists a limit $Y^i$ such that
$$
Y^{i,n}_t\nearrow Y^i_t\leq \hat{Y}_t,
$$
for all $i$, for all $t$, $ \mathbb{P} $-a.s. By dominated convergence theorem we also have that $\|Y^{i,n}-Y^i\|_{L^{2,\beta}(A)}\rightarrow 0$ as $n\rightarrow\infty$.\\
We have the following bound for the norms, which can be obtained by applying the Ito Formula to $e^{\beta A_s}|Y_s|^2$ (see \cite[Proposition 3.1]{foresta2017optimal} for the full details)
\begin{multline}
\label{eq:sequence_bound}
\|Y^{i,n}\|_{L^{2,\beta}(A)}+\|Y^{i,n}\|_{L^{2,\beta}(W)}\|U^{i,n}\|_{L^{2,\beta}(p)}+\|Z^{i,n}\|_{L^{2,\beta}(W)}\leq \eval\left[e^{\beta A_T}|\xi|^2\right]\\+\evals{\int_0^Te^{\beta A_s}|g_s^i|^2ds}+\evals{\int_0^Te^{\beta A_s}|f_s^i|^2dA_s}+\evals{\sup\limits_{t\in [0,T]}e^{\beta A_t}|Y_t^{i,n}|^2}
\\+\evals{\sup\limits_{t\in [0,T]}e^{(\beta+\delta) A_t}|\max\limits_{j\in\mathbb{A}_i}(Y_t^{j,n-1}-C_t(i,j))|^2}.
\end{multline}

Now since $Y^{i,n}$ is a nondecreasing sequence, it holds that
$$
|Y^{i,n}_t|\leq \max\lbrace|Y^{i,0}_t|,|\hat{Y}_t|\rbrace.
$$
Taking into account proposition \ref{swit:Y0andYhat}, we have that
$$
\evals{\sup_{t\in[0,T]} e^{\hat{\beta} A_t}\max\lbrace|Y^{i,0}_t|,|\hat{Y}_t|\rbrace}
$$
and thus the last term in \eqref{swit:Y0andYhat} is bounded by a constant independent of $n$. 

The sequence $(U^{i,n},Z^{i,n})$ is bounded in norm thanks to \eqref{eq:sequence_bound}. Then we can apply proposition \ref{prop:monotonic_point} and obtain the existence of $(U^i,Z^i,K^i) \in L^{2,\beta}(p) \times L^{2,\beta}(W) \times \mathcal{I}^2$ such that:
\begin{equation}
\begin{cases}
Y_t^i=\xi^i+\int_t^Tf_s^idA_s+\int_t^Tg_s^ids+\int_t^T\int_EU_s^i(e)(\rho_s^i(e)-1)\phi_s(de)dA_s\\
\qquad -\int_t^TU^i(e)q(dsde)-\int_t^TZ_s^idW_s+K^i_T-K^i_t\\
Y_t^i\geq\max\limits_{j\in \mathbb{A}_i}(Y_t^j-C_t(i,j)
\end{cases}
\end{equation}

Now consider, for $i\in\mathcal{J}$, the RBSDE with known obstacle

\begin{equation}
\label{eq:system_given}
\begin{cases}
\bar{Y}_t^i=\xi^i+\int_t^Tf_s^idA_s+\int_t^Tg_s^ids+\int_t^T\bar{U}_s^i(e)(\rho_s^i(e)-1)\phi_s(de)dA_s\\
\qquad -\int_t^T\bar{U}_s^i(e)q(dsde)-\int_t^TZ_s^idW_s+\bar{K}_T^i-\bar{K}_t^i\\
\bar{Y}_t^i\geq \max\limits_{j\in \mathbb{A}_i}(Y^j_t-C_t(i,j))\\
\int_0^T(\bar{Y}_t^i- \max\limits_{j\in \mathbb{A}_i}(Y^j_t-C_t(i,j)))d\bar{K}_t^i=0.
\end{cases}
\end{equation}
The solution $(\bar{Y},\bar{U},\bar{Z},\bar{K})$ exists thanks to theorem \cite[Theorem 4.1]{foresta2017optimal}.
Thanks to the comparison result \ref{otherBSDE_comparison}, we have that $Y_t^{i,n}\leq \bar{Y}_t^i$ for all $n$, and thus
\begin{equation}
\label{eq:no_bar_smaller_than_bar}
Y_t^{i}\leq \bar{Y}_t^i.
\end{equation}
We want to prove the reverse inequality, which is a bit more involved.
To this end we will use a couple of lemmas. Fix $\bar{\epsilon}$ such that $3+(M+\bar {\epsilon})^4<\eta$. Here we consider again, for all $i\in\mathcal{J}$ and for all $\epsilon<\bar{\epsilon}$, the probability $\mathbb{P}^{\rho^i+\epsilon}$ introduced as 
$$
\frac{\mathbb{P}^{\rho^i+\epsilon}}{d\mathbb{P}}=\mathcal{E}(\rho^i+\epsilon)
$$
For each $i\in\mathcal{J}$ define 
\begin{equation}
\eta^i_t=\xi^i\indb{t\geq T}+\int_{0}^{t\wedge T} f^i_sdA_s+\int_{0}^{t\wedge T} g^i_sds+\max\limits_{j\in \mathbb{A}_i}(Y^j_t-C_t(i,j))\indb{t<T}.
\end{equation}
\begin{lemma}
	\begin{equation}
	\bar{Y}^i_t+\int_0^tf_s^idA_s+\int_0^tg_s^ids+\epsilon\essup_{\tau\geq t}\eval^{\rho^i+\epsilon}\left[\int_t^\tau\int_E\bar{U}_s^i(e)\phi_s(de)dA_s\right.\left|\vphantom{\int_t^\tau}\mathcal{F}_t\right]
	\end{equation}
	is the $\mathbb{P}^{\rho^i+\epsilon}$ Snell's envelope of $\eta^i$.
\end{lemma}

\begin{proof}
	The reasoning is standard for the connection between reflected BSDE and optimal stopping, minus the fact that we have to consider a different probability. By adding $\epsilon\int_t^\tau\int_E\bar{U}_s^i(e)\phi_s(de)dA_s$ to both sides of the equation solved by $\bar{Y}^i$, considered between $t$ and a generic $\tau\geq t$ we obtain
	\begin{multline*}
	\bar{Y}_t^i+\epsilon\int_t^\tau\int_E\bar{U}_s^i(e)\phi_s(de)dA_s=\bar{Y}_\tau^i+\int_t^\tau f_s^idA_s+\int_t^\tau g_s^ids\\+\int_t^\tau\bar{U}_s^i(e)(\rho_s^i(e)+\epsilon-1)\phi_s(de)dA_s
	-\int_t^\tau\bar{U}_s^i(e)q(dsde)-\int_t^\tau\bar{Z}_s^idW_s+\bar{K}_\tau^i-\bar{K}_t^i.
	\end{multline*}
	Since $\bar{K}$ is increasing, and $\bar{Y}^i_\tau\geq \xi^i\indb{\tau\geq T}+\max\limits_{j\in \mathbb{A}_i}(Y^j_\tau-C_\tau(i,j))\indb{\tau< T}$, we have
	
	\begin{multline*}
	\bar{Y}_t^i+\epsilon\int_t^\tau\int_E\bar{U}_s^i(e)\phi_s(de)dA_s+\int_0^tf_s^idA_s+\int_0^tg_s^ids\geq\\ \xi^i+\max\limits_{j\in \mathbb{A}_i}(Y^j_\tau-C_\tau(i,j))\indb{\tau< T}+\int_0^\tau f_s^idA_s+\int_0^\tau g_s^ids\\+\int_t^\tau\bar{U}_s^i(e)(\rho_s^i(e)+\epsilon-1)\phi_s(de)dA_s-\int_t^\tau \bar{Z}_s^i
	-\int_t^\tau\bar{U}_s^i(e)q(dsde).
	\end{multline*}
	Being a $\mathbb{P}^{\rho^i+\epsilon}$-martingale, by taking expectation under $\mathbb{P}^{\rho^i+\epsilon}$ the term
	$$
	+\int_t^\tau\bar{U}_s^i(e)(\rho_s^i(e)+\epsilon-1)\phi_s(de)dA_s
	-\int_t^\tau\bar{U}_s^i(e)q(dsde)-\int_t^\tau \bar{Z}_s^i dW_s
	$$
	disappears and we obtain
	\begin{multline*}
	\bar{Y}_t^i+\epsilon\econdm{\int_t^\tau\int_E\bar{U}_s^i(e)\phi_s(de)dA_s}{\rho^i+\epsilon}
	+\int_0^tf_s^idA_s +\int_0^tg_s^i ds
	\\\geq \econdm{\xi^i\indb{\tau\geq T}+\max\limits_{j\in \mathbb{A}_i}(Y^j_\tau-C_\tau(i,j))\indb{\tau<T}+\int_0^\tau f_s^idA_s+\int_0^\tau g_s^ids}{\rho^i+\epsilon}.
	\end{multline*}
	By taking the $\essup_{\mathbf{a}\in\mathcal{A}_t^i}$ over $\tau\geq t$ we obtain
	
	\begin{multline}
	\label{eq:rho_epsilon_snell_ge}
	\bar{Y}_t^i+\epsilon\essup_{\tau\geq t}\econdm{\int_t^\tau\int_E\bar{U}_s^i(e)\phi_s(de)dA_s}{\rho^i+\epsilon}
	\\+\int_0^tf_s^idA_s+\int_0^tg_s^ids
	\geq R^{\rho^i+\epsilon}(\eta^i),
	\end{multline}
	where $R^{\rho^i+\epsilon}(\eta^i)$ indicates the Snell envelope of $\eta^i$ under the probability $\mathbb{P}^{\rho^i+\epsilon}$.
	For the reverse inequality, consider the stopping time
	\begin{equation}
	D_t^\delta=\inf\left\lbrace s\geq t : \bar{Y}_s\leq \max\limits_{j\in \mathbb{A}_i}(Y^j_s-C_s(i,j)) +\delta \right\rbrace \wedge T.
	\end{equation}
	We repeat the reasoning above, but in this case $ \bar{K}_{D_t^\delta}=\bar{K}_t $ and $$\bar{Y}_{D_t^\delta}\leq \xi^i\indb{D_t^\delta\geq T}+\max\limits_{j\in \mathbb{A}_i}(Y^j_{D_t^\delta}-C_{D_t^\delta}(i,j))\indb{D_t^\delta<T},$$
	and we obtain
	\begin{multline*}
	\bar{Y}_t^i+\epsilon\econdm{\int_t^{D_t^\delta}\int_E\bar{U}_s^i(e)\phi_s(de)dA_s}{\rho^i+\epsilon}
	+\int_0^tf_s^idA_s+\int_0^tg_s^ids
	\\\leq \econdm{\xi^i\indb{{D_t^\delta}\geq T}+\max\limits_{j\in \mathbb{A}_i}(Y^j_{D_t^\delta}-C_{D_t^\delta}(i,j))+\int_0^{D_t^\delta} f_s^idA_s+\int_0^{D_t^\delta}g_s^ids}{\rho^i+\epsilon}+\delta.
	\end{multline*}
	By taking the $\essup$ and sending $\delta$ to zero we obtain
	\begin{multline*}
	\bar{Y}_t^i+\epsilon\essup_{\tau\geq t}\econdm{\int_t^\tau\int_E\bar{U}_s^i(e)\phi_s(de)dA_s}{\rho^i+\epsilon}
	\\+\int_0^tf_s^idA_s
	+\int_0^tg_s^ids \leq R^{\rho^i+\epsilon}(\eta^i),
	\end{multline*}
	that together with \eqref{eq:rho_epsilon_snell_ge} completes the proof.
\end{proof}

On the other hand we have

\begin{lemma}
	The process
	$$
	{Y}^i_t+\int_0^tf_s^idA_s+\int_0^tg_s^i ds+\epsilon\essup_{\tau\geq t}\eval^{\rho^i+\epsilon}\left[\int_t^\tau\int_E{U}_s^i(e)\phi_s(de)dA_s\right.\left|\vphantom{\int_t^\tau}\mathcal{F}_t\right]
	$$
	is a $\mathbb{P}^{\rho^i+\epsilon}$-supermartingale such that
	$$
	{Y}^i_t+\int_0^tf_s^idA_s+\int_0^tg_s^i ds+\epsilon\essup_{\tau\geq t}\eval^{\rho^i+\epsilon}\left[\int_t^\tau\int_E{U}_s^i(e)\phi_s(de)dA_s\right.\left|\vphantom{\int_t^\tau}\mathcal{F}_t\right]\geq \eta^i_t
	$$
\end{lemma}

\begin{proof}
	Clearly we have
	\begin{multline*}
	Y_t^i+\int_0^tf_s^idA_s+\int_0^tg_s^i ds\geq\\ \xi^i\indb{t\geq T}+\max\limits_{j\in \mathbb{A}_i}(Y^j_t-C_t(i,j))\indb{t<T}+\int_0^tf_s^idA_s+\int_0^tg_s^i ds=\eta^i_t,
	\end{multline*}
	and since $\epsilon\essup_{\tau\geq t}\eval^{\rho^i+\epsilon}\left[\int_t^\tau\int_E{U}_s^i(e)\phi_s(de)dA_s\right.\left|\vphantom{\int_t^\tau}\mathcal{F}_t\right]\geq 0$ we obtain the second property, that is
	$$
	Y_t^i+\int_0^tf_s^idA_s+\int_0^tg_s^i ds+\epsilon\essup_{\tau\geq t}\eval^{\rho^i+\epsilon}\left[\int_t^\tau\int_E{U}_s^i(e)\phi_s(de)dA_s\right.\left|\vphantom{\int_t^\tau}\mathcal{F}_t\right]\geq \eta^i_t
	$$
	To show it is a supermartingale consider the equation solved by $Y^i$ between $t$ and a generic stopping time $\tau$, where we add to both sides of the equation the quantity $\epsilon\int_t^\tau\int_E{U}_s^i(e)\phi_s(de)dA_s$ and take conditional expected value under the measure $\mathbb{P}^{\rho^i+\epsilon}$:
	\begin{multline*}
	Y_t^i+\int_0^tf_s^idA_s+\int_0^tg_s^i ds+\epsilon\eval^{\rho^i+\epsilon}\left[\int_t^\tau\int_E{U}_s^i(e)\phi_s(de)dA_s\right.\left|\vphantom{\int_t^\tau}\mathcal{F}_t\right]\\=\econdm{Y_\tau+\int_0^\tau f_s^idA_s+\int_0^\tau g_s^i ds+K^i_\tau}{\rho^i+\epsilon}-K_t^i.
	\end{multline*}
	By taking $\essup_{\tau\geq t}$ on both sides it becomes
	\begin{multline*}
	Y_t^i+\int_0^tf_s^idA_s+\int_0^tg_s^i ds+\epsilon\essup_{\tau\geq t}\eval^{\rho^i+\epsilon}\left[\int_t^\tau\int_E{U}_s^i(e)\phi_s(de)dA_s\right.\left|\vphantom{\int_t^\tau}\mathcal{F}_t\right]\\=\essup_{\tau\geq t}\econdm{Y_\tau+\int_0^\tau f_s^idA_s+\int_0^\tau g_s^i ds+K^i_\tau}{\rho^i+\epsilon}-K_t^i.
	\end{multline*}
	The term $$\essup_{\tau\geq t}\econdm{Y_\tau+\int_0^\tau f_s^idA_s+K^i_\tau}{\rho^i+\epsilon}$$ is a $\mathbb{P}^{\rho^i+\epsilon}$-supermartingale (it is a Snell envelope) from which we are subtracting an increasing process. The right hand side then a $\mathbb{P}^{\rho^i+\epsilon}$-supermartingale and so is the left hand side.
\end{proof}

With this two lemmas in hand, we can show that $\mathbb{P}$-a.s. $\bar{Y}_t\leq Y_t$.
\begin{prop}
	It holds that $\mathbb{P}$-a.s. 
	$$
	\bar{Y}_t\leq Y_t.
	$$
\end{prop}

\begin{proof}
	Since the Snell envelope of a process is the smallest super-martingale that dominates the process, by the lemmas above we have that $\mathbb{P}^{\rho^i+\epsilon}$-a.s.
	\begin{multline*}
	\bar{Y}^i_t+\int_0^tf_s^idA_s+\int_0^tg_s^i ds+\epsilon\essup_{\tau\geq t}\eval^{\rho^i+\epsilon}\left[\int_t^\tau\int_E\bar{U}_s^i(e)\phi_s(de)dA_s\right.\left|\vphantom{\int_t^\tau}\mathcal{F}_t\right]\\
	\leq {Y}^i_t+\int_0^tf_s^idA_s+\int_0^tg_s^i ds+\epsilon\essup_{\tau\geq t}\eval^{\rho^i+\epsilon}\left[\int_t^\tau\int_E{U}_s^i(e)\phi_s(de)dA_s\right.\left|\vphantom{\int_t^\tau}\mathcal{F}_t\right].
	\end{multline*}
	
	Since $\mathbb{P}^{\rho^i+\epsilon}\sim\mathbb{P}$, this holds also $\mathbb{P}$-a.s.
	We must now show that
	$$
	\epsilon\essup_{\tau\geq t}\eval^{\rho^i+\epsilon}\left[\int_t^\tau\int_E\bar{U}_s^i(e)\phi_s(de)dA_s\right.\left|\vphantom{\int_t^\tau}\mathcal{F}_t\right]\rightarrow 0
	$$
	for $\epsilon\rightarrow 0$, and that the same holds for the integral of $U^i$. This can be done as in the proof of proposition \ref{swit_existence_picard_iteration}.

	The same applies to the integral of $U^i$.
	This tells us that for all $t$
	$$
	\bar{Y}_t\leq Y_t\quad \mathbb{P}\text{-a.s.}
	$$
	Since both processes are càdlàg, this holds $\mathbb{P}\text{-a.s.}$ for all $t$.
\end{proof}

By combining this lemma and \eqref{eq:no_bar_smaller_than_bar}, we obtain that $Y$ is indistinguishable from $\bar{Y}$. Once we know this is, it means that
\begin{multline*}
\int_t^T\int_EU_s^i(e)(\rho^i_s(de)-1)\phi_s(de)dA_s-\int_t^TU_s^i(e)q^i(dsde)-\int_t^TZ_s^idW_s+K_T^i-K_t^i\\
=\int_t^T\int_E\bar{U}_s^i(e)(\rho^i_s(de)-1)\phi_s(de)dA_s-\int_t^T\bar{U}_s^i(e)q^i(dsde)-\int_t^T\bar{Z}_s^idW_s+\bar{K}_T^i-\bar{K}_t^i.
\end{multline*}
Then, reasoning as in the proof of uniqueness \ref{swit_uniqueness}, it is possible to show that $U=\bar{U}$ since the inaccessible jumps of both sides are the same, $Z=\bar{Z}$ by considering the predictable bracket against $\int_0^t(Z_s-\bar{Z}_s)dW_s$ and thus that $K_t=\bar{K}_t$.

$(Y^i,U^i,Z^i,K^i)_{i\in\mathcal{J}}$ is almost the solution to system \eqref{eq:system}, in the sense that is solves
\begin{equation}
\label{eq:system_leftlimits}
\begin{cases}
Y_t^i=\xi^i+\int_t^Tf_s^idA^i_s+\int_t^Tg_s^i ds+\int_t^T\int_EU_s^i(e)(\rho^i_s(de)-1)\phi_s(de)dA_s\\
\qquad -\int_t^TU_s^i(e)q^i(dsde)-\int_t^TZ_s^idW_s+K_T^i-K_t^i\\
Y_t^i\geq \max\limits_{j\in \mathbb{A}_i}(Y^j_t-C_t(i,j))\\
\int_0^T(Y_{t^-}^i- \max\limits_{j\in \mathbb{A}_i}(Y^j_{t^-}-C_t(i,j)))dK_t^i=0.
\end{cases}
\end{equation}

We only have to show that the $ K^i $ is continuous. This also means that $Y^i$ jumps are totally inaccessible. This is a crucial property for the existence of optimal stopping times. It also means that $Y^i$ are Upper Semi Continuous in Expectation (USCE see \cite{kobylanski2012optimal} )
\begin{prop}
	\label{swit:Kconti}
	The processes $K^i$ are continuous.
\end{prop}

\begin{proof}
	Since $(K^i)_{i\in\mathcal{J}}$ are predictable process, their jump times are predictable stopping times. Also all jumps are non-negative since $(K^i)_{i\in\mathcal{J}}$ are increasing. Suppose then there exists $j_1$ and $\tau$ such that $\Delta K^{j_1}_\tau>0$, with $\tau$ a predictable stopping time. Since the martingale part has only totally inaccessible jumps, it holds that
	$$
	\Delta Y_\tau^{j_1}=-\Delta K^{j_1}_\tau<0.
	$$
	Thanks to the Skorohod condition in \ref{eq:system_leftlimits}, it holds that
	$$
	Y_{\tau^-}^{j_1}=\max\limits_{k\in\mathbb{A}_{j_1}}(Y_{\tau^-}^k-C_t(j_1,k)).
	$$
	This means that for some index $j_2$,
	$$
	Y_{\tau^-}^{j_2}-C_\tau(j_1,j_2)=Y_{\tau^-}^{j_1}\geq Y_\tau^{j_1}\geq Y_{\tau}^{j_2}-C_\tau(j_1,j_2)
	$$
	where the first inequality is because $\Delta Y^{j_1}_\tau<0$. We deduce $\Delta Y^{j_2}_\tau<0$ and $\Delta K_\tau^{j_2}>0$. Then again
	$$
	Y_{\tau^-}^{j_2}=\max\limits_{k\in\mathbb{A}_{j_2}}(Y_{\tau^-}^k-C_t(j_2,k)),
	$$
	and there is $j_3$ such that 
	$$
	Y_{\tau^-}^{j_3}-C_\tau(j_2,j_3)=Y_{\tau^-}^{j_2}\geq Y_\tau^{j_2}\geq Y_{\tau}^{j_3}-C_\tau(j_2,j_3)
	$$
	and $-\Delta Y_\tau^{j_3}=\Delta K_\tau^{j_3}>0$. Since the set of indexes is finite, by iterating this procedure it is possible to find a finite sequence of numbers $j_1,\dots,j_p=j_1$ such that
	$$
	Y_{\tau^-}^{j_1}=Y_{\tau^-}^{j_2}-C_\tau(j_1,j_2),\;Y_{\tau^-}^{j_2}=Y_{\tau^-}^{j_3}-C_\tau(j_2,j_3),\dots,Y_{\tau^-}^{j_{p-1}}=Y_{\tau^-}^{j_p}-C_\tau(j_{p-1},j_p).
	$$
	This last set of equalities implies
	\begin{equation}
	\label{eq:costs_sum_zero}
	\sum_{k}^{p-1}C_\tau(j_k,j_{k+1})=0.
	\end{equation}
	Using condition \eqref{ass:no_free_cycles}, we can write
	\begin{align*}
	C_\tau(j_1,j_2)+C_\tau(j_2,j_3)&>C_\tau(j_1,j_3)\\
	C_\tau(j_1,j_2)+C_\tau(j_2,j_3)+C_\tau(j_3,j_4)&>C_\tau(j_1,j_3)+C_\tau(j_3,j_4)>C_\tau(j_1,j_4)\\
	&\vdots\\
	\sum_{k=1}^{p-1}C_\tau(j_{k},j_{k+1})&>C_\tau(j_1,j_p)=C_\tau(j_1,j_1)=0
	\end{align*}
	which contradicts \eqref{eq:costs_sum_zero}. Thus $K^{j_1}_\tau=0$ and $K^i$ are continuous.
\end{proof}
Since $K^i$ are continuous, there is no need for limits in the Skorohod condition and we have a solution for \eqref{eq:system}. Note that all the $Y^i$ are USCE and have only inaccessible jumps.

\section{Verification theorem}
\label{sec_swit_verification}
Now that we have established existence of a solution to the system, we can use it to represent the value function. The idea is the following. We ``glue together" the solutions to the system in accordance with strategies: given a strategy in $\mathcal{A}_t^i$, we start with $Y_t^i$ and consider its dynamic between $t$ and the first switching time. Then we switch to another $Y_t^j$ according to the strategy. This way we obtain a ``switched process" that contains the rewards minus the cost, plus a $\prob$-martingale part and the integral of the $U$ parts of the equations against the $\rho$. This last piece will be altogether a martingale under the probability induced by the strategy, thus by taking expected value under that measure we obtain the gain for this strategy.

First, define 
\begin{defi}
	Given a strategy $a$ in $ \mathcal{A}_t^i $, the cumulated switching cost $D_s^a$ is defined as
	\begin{equation}
	D_s^\mathbf{a}=\sum_{n\geq 1}C_{\theta_n}(\alpha_{n-1},\alpha_n)\indb{\theta_n\leq t}\qquad D_T^\mathbf{a}=\lim\limits_{s\rightarrow T}D_s^\mathbf{a}.
	\end{equation}
\end{defi}

$D_s^a$ is a càdlàg adapted increasing process.
The value function is rewritten as 

\begin{equation}
\label{eq:value_func_with_increasing}
v(t,i)=\essup_{\mathbf{a}\in\mathcal{A}_t^i}J(t,i,\mathbf{a})=\eval^{\mathbf{a}}\left[\left.\xi^{a_T}+\int_t^Tf_s^{a_s}dA_s+\int_t^T g_s^{a_s}ds-D_T^\mathbf{a}\right|\mathcal{F}_t\right].
\end{equation}

We have the main result
\begin{teo}
	\label{swit:representation_property}
	Let $(Y^i,U^i,Z^i,K^i)_{i\in\mathcal{J}}$ be the solution to the system \eqref{eq:system}. Then
	\begin{equation}
	\label{eq:representation}
	Y_t^i=\essup\limits_{\mathbf{a}\in\mathcal{A}_t^i}J(t,i,\mathbf{a})=v(t,i)\;\prob\text{-a.s.}
	\end{equation}
	Moreover, the strategy $\mathbf{a}^*=(\theta^*_n,\alpha^*_n)$ defined as $(\theta_0^*,\alpha_0^*)=(t,i)$ and
	\begin{gather}
	\label{eq:optimal_times}\theta^*_n=\inf\left\lbrace s\geq\theta_{n-1}^* : Y^{\alpha^*_{n-1}}_s=\max\limits_{j\in\mathbb{A}_{\alpha^*_{n-1}}}(Y^j_s-C_s(\alpha^*_{n-1},j))\right\rbrace\\
	\label{eq:optimal_choices}\alpha^*_n=\arg\max\limits_{j\in\mathbb{A}_{\alpha^*_{n-1}}}\left(Y^k_{\theta_n^*}-C_{\theta^*_n}(\alpha^*_{n-1},j)\right)
	\end{gather}
	is an optimal strategy.
\end{teo}

\begin{proof}
	\textit{Step 1}. We first show it in the case where the $\rho^i$ also satisfy $0<c\leq \rho_t^i(e)$ for some constant $c$. 
	Consider $(Y^i,U^i,Z^i,K^i)_{i\in\mathcal{J}}$ be the solution to the system \eqref{eq:system}. Let $\mathbf{a}\in\mathcal{A}_t^i$, and for this $\mathbf{a}$ define
	$$
	\hat{K}_T^\mathbf{a}=K_{\theta_1}^i-K_{t}^i+\sum_{n\geq1}^{}\left(K^{\alpha_n}_{\theta_{n+1}}-K^{\alpha_n}_{\theta_n}\right)
	$$
	and
	$$
	U_r^\mathbf{a}=\sum_{n\geq 0}U_r^{\alpha_n}\indb{\theta_n<r \leq \theta_{n+1}} \quad Z_r^\mathbf{a}=\sum_{n\geq 0}Z_r^{\alpha_n}\indb{\theta_n<r \leq \theta_{n+1}}
	$$
	Now rewrite the equation for $Y^i$ between $t$ and $\theta_1$:
	\begin{align*}
	Y_t^i&=Y_{\theta_1}^i+\int_t^{\theta_1}f_s^idA_s+\int_t^{\theta_1}g_s^ids+\int_t^{\theta_1}\int_EU_s^i(e)(\rho_s^i-1)\phi_s(de)dA_s\\
	&-\int_t^{\theta_1}\int_EU_s^i(e)q(dsde)-\int_t^{\theta_1}Z_s^idW_s+K_{\theta_1}^i-K_t^i\\
	&\geq \left(Y_{\theta_1}^{\alpha_1}-C_{\theta_1}(i,\alpha_1)\right)\indb{\theta_1<T}+\xi^{\alpha_0}\indb{\theta_1=T}+\int_t^{\theta_1}f_s^{a_s}dA_s+\int_t^{\theta_1}g_s^{a_s}ds\\
	&+\int_t^{\theta_1}\int_EU_s^{\mathbf{a}}(e)(\rho_s^{\mathbf{a}}-1)\phi_s(de)dA_s
	-\int_t^{\theta_1}\int_EU_s^{\mathbf{a}}(e)q(dsde)\\
	&-\int_t^{\theta_1}Z_s^{\mathbf{a}}dW_s+K_{\theta_1}^i-K_t^i\\
	&=Y_{\theta_2}^{\alpha_1}\indb{\theta_1<T}+\int_t^{\theta_2}f_s^{a_s}dA_s
	+\int_t^{\theta_2}g_s^{a_s}ds\\ &+\int_t^{\theta_2}\int_EU_s^\mathbf{a}(e)(\rho_s^{\mathbf{a}}-1)\phi_s(de)dA_s-\int_t^{\theta_2}\int_EU_s^{\mathbf{a}}(e)q(dsde)+\\
	&+\int_t^{\theta_2}Z_s^\mathbf{a}dW_s+
	(K_{\theta_1}^i-K_t^i)+(K_{\theta_2}^{\alpha_1}-K_{\theta_1}^{\alpha_1})
	-C_{\theta_1}(i,\alpha_1)\indb{\theta_1<T}+\xi^{\alpha_0}\indb{\theta_1=T},
	\end{align*}
	where we used first the barrier condition and then the equation for $Y_{\theta_1}^{\alpha_1}$ between $\theta_1$ and $\theta_2$. This process can be repeated until we obtain
	\begin{multline*}
	Y_t^i\geq \xi^{a_T}+\int_t^Tf_s^{a_s}dA_s+\int_t^Tg_s^{a_s}ds+\int_t^T\int_EU_s^{\mathbf{a}}(\rho_s^\mathbf{a}-1)\phi_s(de)dA_s\\
	-\int_t^T\int_EU_s^{\mathbf{a}}(e)q(dsde)-\int_t^TZ_s^\mathbf{a}dW_s+\hat{K}_T^a-D_T^a,
	\end{multline*}
	which ends thanks to the fact that $a$ is an admissible strategy and thus $\mathbb{P}(\theta_n<T \;\forall n\geq 0)=0$.
	This can be rewritten as (forgetting about the non-negative $\hat{K}_T^a$)
	\begin{multline}
	\label{swit_Y_greater_than_a}
	Y_t^i\geq \xi^{a_T}+\int_t^Tf_s^{a_s}dA_s+\int_t^Tg_s^{a_s}ds-\int_t^TZ_s^\mathbf{a}dW_s\\
	-\int_t^T\int_EU_s^{\mathbf{a}}(e)(p(dsde)-\rho_s^\mathbf{a}(e)\phi_s(de)dA_s)-D_T^a.
	\end{multline}
	Now, by taking $\mathbb{E}^{\mathbf{a}}$ expectation, we have that
	\begin{multline*}
	Y_t^i\geq \econdm{\xi^{a_T}+\int_t^Tf_s^{a_s}dA_s+\int_t^Tg_s^{a_s}ds-D_T^a}{\mathbf{a}}\\-\econdm{\int_t^T\int_EU_s^{\mathbf{a}}(e)(p(dsde)-\rho_s^\mathbf{a}(e)\phi_s(de)dA_s)+\int_t^TZ_s^\mathbf{a}dW_s}{\mathbf{a}},
	\end{multline*}
	now the term $\int_t^T\int_EU_s^{\mathbf{a}}(e)(p(dsde)-\rho_s^\mathbf{a}(e)\phi_s(de)dA_s)-\int_t^TZ_s^\mathbf{a}dW_s$ is a $\mathbb{P}^\mathbf{a}$-martingale, as stated by proposition \ref{prop:still_mg_and_square}. Thus we have:
\begin{equation}
\label{swit_Y_greater_than_a_expected_value}
Y_t^i\geq \econdm{\xi^{a_T}+\int_t^Tf_s^{a_s}dA_s-D_T^a}{\mathbf{a}} \quad \prob^{\mathbf{a}}\text{-a.s.}
\end{equation}
	To prove that $Y^i$ is indeed the $\essup$, consider the strategy $\mathbf{a}^*$ defined above. We first prove that it is indeed an admissible strategy, by showing that $\prob(\theta^*_n<T\; \forall n\geq 0)=0$.
	Assume, as done in \cite{hamadene2015systems}, that this does not hold and $\prob(\theta^*_n<T\; \forall n\geq 0)>0$. By the definition of $\mathbf{a^*}$ this means that
	
	\begin{equation*}
	\prob\left(Y_{\theta^*_{n+1}}^{\alpha_n^*}=Y_{\theta^*_{n+1}}^{\alpha_{n+1}^*}-C_{\theta^*_{n+1}}(\alpha_{n}^*,\alpha_{n+1}^*),\; \alpha_{n+1}^*\in\mathbb{A}_{a^*_n}, \;\forall n\geq 0\right)>0.
	\end{equation*}
	Since $\mathcal{J}$ is finite, there exists a loop $i_0,i_1,\dots,i_k,i_0$ of elements of $\mathcal{J}$ and a subsequence $n_q(\omega)_{q\geq 0}$ such that
	
	\begin{equation}
	\label{swit_prob_to_take_lim}
	\prob\left(Y_{\theta^*_{n_{q+l}}}^{i_l}=Y_{\theta^*_{n_{q+l}}}^{i_{l+1}}-C_{\theta^*_{n_{q+l}}}(i_l,i_l+1),\; l=0,\dots,k ;\;i_{k+1}=i_0 \;\forall q\geq 0\right)>0.
	\end{equation}
	Consider now $\theta^*=\lim\theta_n^*$ and the set $\Theta=\lbrace\theta^*_n<\theta^*, \;\forall n\geq 0\rbrace$.
	Thanks to property \ref{ass:no_free_cycles} we know that
	$$\prob(\lbrace\theta^*<T\rbrace\cap\Theta^c)=0.$$
	Indeed on $\Theta^c$ we know that for some $\bar{n}$ we have $\theta_n^*=\theta^*$ for all $n\geq\bar{n}$. This would mean that we keep switching over a cycle with cost zero at $\theta^*$ (see proposition \ref{swit:Kconti} or remark \ref{swit_another_no_free_loop}). Indeed we would have
	$$
	Y_{\theta^*}^{j_0}=Y_{\theta^*}^{j_1}-C_{\theta^*}(j_0,j_1),\;Y_{\theta^*}^{j_1}=Y_{\theta^*}^{j_2}-C_{\theta^*}(j_1,j_2),\;\dots,\; Y_{\theta^*}^{j_k}=Y_{\theta^*}^{j_0}-C_{\theta^*}(j_k,j_0),
	$$
	and this again contradicts the no free loop property.
	Now this tells us that $\theta^*$ is not a totally inaccessible stopping time (see \cite[Chapter 4, 79]{dellacherie75probpotent} or \cite[Chapter 3,3]{sheng1998semimartingale}),  since it must hold that $\prob(\Theta)>0$.
	Since the jumps of the $Y^i$ are totally inaccessible, this means that there is no jump at $\theta^*$ and thus we can take the limit in \eqref{swit_prob_to_take_lim} obtaining
	
	\begin{equation*}
	\prob\left(Y_{\theta^*}^{i_l}=Y_{\theta^*}^{i_{l+1}}-C_{\theta^*}(i_l,i_l+1),\; l=0,\dots,k ;\;i_{k+1}=i_0 \;\forall q\geq 0\right)>0.
	\end{equation*}
	This can be rewritten as
	$$
	\prob\left(\sum_{l=0}^{k}C_{\theta^*}(i_l,i_{l+1})=0;\quad i_{k+1}=i_0\right)>0
	$$
	which contradicts the non free loop property (see again remark \ref{swit_another_no_free_loop}).
	
	Now that we know that $a^*\in\mathcal{A}_t^i$ , we write as before,
	\begin{align*}
	Y_t^i&=Y_{\theta^*_1}^i+\int_t^{\theta^*_1}f_s^{a^*_s}dA_s+\int_t^{\theta^*_1}g_s^{a^*_s}ds+\int_t^{\theta_1^*}\int_EU_s^{\mathbf{a}^*}(e)(\rho_s^{\mathbf{a}^*}-1)\phi_s(de)dA_s\\
	&-\int_t^{\theta^*_1}\int_EU_s^{\mathbf{a}^*}(e)q(dsde)-\int_t^{\theta_1}Z_s^{\mathbf{a^*}}dW_s\\
	&= \left(Y_{\theta_1^*}^{\alpha_1^*}-C_{\theta_1^*}(i,\alpha_1^*)\right)\indb{\theta_1^*<T}+\xi^{\alpha_0^*}\indb{\theta_1^*=T}+\int_t^{\theta_1^*}f_s^{a^*_s}dA_s+\int_t^{\theta^*_1}g_s^{a^*_s}ds\\&+\int_t^{\theta_1^*}\int_EU_s^{\mathbf{a}^*}(e)(\rho_s^{\mathbf{a}^*}-1)\phi_s(de)dA_s
	-\int_t^{\theta_1^*}\int_EU_s^{\mathbf{a}^*}(e)q(dsde)-\int_t^{\theta_1}Z_s^{\mathbf{a^*}}dW_s,
	\end{align*}
	but with the difference that $K^i_{\theta^*_1}-K_t^i=0$ thanks to the Skorohod condition in \eqref{eq:system} and to the way $a^*$ is defined. Repeating this as before, but with equalities, we obtain that 
	\begin{multline*}
	Y_t^i=\xi^{a^*_T}+\int_t^{T}f_s^{a^*_s}dA_s+\int_t^{T}g_s^{a^*_s}ds-D_T^{\mathbf{a}^*}\\-\int_t^T\int_EU_s^{\mathbf{a}^*}(e)(p(dsde)-\rho_s^{\mathbf{a}^*}\phi_s(de)dA_s)-\int_t^TZ_s^{\mathbf{a^*}}dW_s.
	\end{multline*}
	By taking $\prob^{\mathbf{a^*}}$-expected value conditional on $\mathcal{F}_t$ we obtain that
\begin{equation}
\label{swit_Y_star_equal_a_star_expected}
Y_t^i=J(t,i,a^*)\quad \prob^{\mathbf{a^*}}\text{-a.s.}
\end{equation}
Since we assumed that the $\rho^i$ are bounded from below by a constant $c>0$, then the probabilities $\prob^\mathbf{a}$ introduced are equivalent to $\prob$ and relations \eqref{swit_Y_greater_than_a_expected_value} and \eqref{swit_Y_star_equal_a_star_expected} hold also $\prob$-a.s. In that case we have that $Y_t^i=v(t,i)$ $\prob$-a.s. by combining the two relations.

\textit{Step 2.} The general case where the $\rho^i$ can touch zero is more complicated, but both in the cases of a generic $\mathbf{a}$ and $\mathbf{a^*}$ we can proceed in the same way. Take a step backwards and consider the relation \eqref{swit_Y_greater_than_a}

\begin{multline*}
Y_t^i\geq \xi^{a_T}+\int_t^Tf_s^{a_s}dA_s+\int_t^Tg_s^{a_s}ds-\int_t^TZ_s^\mathbf{a}dW_s\\
-\int_t^T\int_EU_s^{\mathbf{a}}(e)(p(dsde)-\rho_s^\mathbf{a}(e)\phi_s(de)dA_s)-D_T^a.
\end{multline*}
Consider now $\bar{\epsilon}$ such that $3+(M+\epsilon)^4<\eta$. For all $0<\epsilon<\bar{\epsilon}$ we add $\epsilon\int_t^T\int_E U_s^{\mathbf{a}}\phi_s(de)dA_s$ to both sides of the previous relation, obtaining

\begin{multline*}
Y_t^i+\epsilon\int_t^T\int_E U_s^{\mathbf{a}}\phi_s(de)dA_s \geq \xi^{a_T}+\int_t^Tf_s^{a_s}dA_s+\int_t^Tg_s^{a_s}ds-\int_t^TZ_s^\mathbf{a}dW_s\\
-\int_t^T\int_EU_s^{\mathbf{a}}(e)(p(dsde)-(\rho_s^\mathbf{a}(e)+\epsilon)\phi_s(de)dA_s)-D_T^a.
\end{multline*}
We denote by $\prob^{\mathbf{a}+\epsilon}$ the probability induced by the kernel $\rho^\mathbf{a}+\epsilon$, and we have $\prob^{\mathbf{a}+\epsilon}\sim\prob$. Then by taking expectation under $\prob^{\mathbf{a}+\epsilon}$, we obtain that $\prob$-a.s.

\begin{equation}
Y_t^i+\epsilon\econdm{\int_t^T\int_E U_s^{\mathbf{a}}\phi_s(de)dA_s}{\mathbf{a}+\epsilon}\geq \econdm{\xi^{a_T}+\int_t^Tf_s^{a_s}dA_s-D_T^a}{\mathbf{a}+\epsilon} 
\end{equation}
We have already seen that (see proposition \ref{swit_existence_picard_iteration})
\begin{equation}
\label{swit_representation_epsilon_part}
\epsilon\econdm{\int_t^T\int_E U_s^{\mathbf{a}}\phi_s(de)dA_s}{\mathbf{a}+\epsilon}\rightarrow 0 \text{ for } \epsilon\rightarrow 0.
\end{equation}
If we show that $\text{ for } \epsilon\rightarrow 0$
$$
\econdm{\xi^{a_T}+\int_t^Tf_s^{a_s}dA_s-D_T^a}{\mathbf{a}+\epsilon}\rightarrow \econdm{\xi^{a_T}+\int_t^Tf_s^{a_s}dA_s-D_T^a}{\mathbf{a}},
$$
then we obtain that $\prob$-a.s. for all $\mathbf{a}\in\mathcal{A}_t^i$
\begin{equation}
\label{swit_Y_greater_reward_no_low_bound}
Y_t^i\geq J(t,i,\mathbf{a}).
\end{equation}
Denote by $X$ the square integrable random variable $X=\xi^{a_T}+\int_t^Tf_s^{a_s}dA_s-D_T^a$. And consider the difference between the two terms
\begin{align*}
\left|\econdm{X}{\mathbf{a}+\epsilon}-\econdm{X}{\mathbf{a}}\right|&=\left|\econd{\frac{L_T^{\mathbf{a}+\epsilon}}{L_t^{\mathbf{a}+\epsilon}}X}-\econd{\frac{L_T^{\mathbf{a}}}{L_t^{\mathbf{a}}}X}\right|\\
&=\left|\econd{\left(\frac{L_T^{\mathbf{a}+\epsilon}}{L_t^{\mathbf{a}+\epsilon}}-\frac{L_T^{\mathbf{a}}}{L_t^{\mathbf{a}}}\right)X}\right|\\
&\leq \econd{\left(\frac{L_T^{\mathbf{a}+\epsilon}}{L_t^{\mathbf{a}+\epsilon}}-\frac{L_T^{\mathbf{a}}}{L_t^{\mathbf{a}}}\right)^2}^{1/2}\econd{X^2}^{1/2}.
\end{align*}
It is clear that
$$
\left(\frac{L_T^{\mathbf{a}+\epsilon}}{L_t^{\mathbf{a}+\epsilon}}-\frac{L_T^{\mathbf{a}}}{L_t^{\mathbf{a}}}\right)^2\rightarrow 0\quad \prob\text{-a.s.}
$$
so we want to use the dominated convergence theorem to prove that the conditional expected value converges to zero. It holds that
$$\left(\frac{L_T^{\mathbf{a}+\epsilon}}{L_t^{\mathbf{a}+\epsilon}}-\frac{L_T^{\mathbf{a}}}{L_t^{\mathbf{a}}}\right)^2\leq C\left(\frac{L_T^{\mathbf{a}+\epsilon}}{L_t^{\mathbf{a}+\epsilon}}\right)^2+C\left(\frac{L_T^{\mathbf{a}}}{L_t^{\mathbf{a}}}\right)^2  $$
We already know, thanks to proposition \ref{prop:still_mg_and_square}, that the second term is bounded by
$$
\sqrt{\frac{\mathcal{E}((\rho^\mathbf{a})^4)_T}{\mathcal{E}((\rho^\mathbf{a})^4)_t}}e^{\frac{\eta}{2}A_T}
$$
which is integrable
$$
\evals{\sqrt{\frac{\mathcal{E}((\rho^\mathbf{a})^4)_T}{\mathcal{E}((\rho^\mathbf{a})^4)_t}}e^{\frac{\eta}{2}A_T}}\leq C\evals{\frac{\mathcal{E}((\rho^\mathbf{a})^4)_T}{\mathcal{E}((\rho^\mathbf{a})^4)_t}}+C\evals{e^{\eta A_T}}\leq C +C \evals{e^{\eta A_T}}.
$$
For the first term we obtain a similar estimate, we have that $\left(\frac{L_T^{\mathbf{a}+\epsilon}}{L_t^{\mathbf{a}+\epsilon}}\right)^2$ can be rewritten explicitly as

\begin{align*}
&\prod_{t<T_n\leq T}(\rho_{t_n}(\xi_n)+\epsilon)^2e^{2\int_t^T\int_E (1-\rho_s(e)-\epsilon)\phi_s(de)dA_s}\\
&\leq \prod_{t<T_n\leq T}(\rho_{t_n}(\xi_n)+\bar{\epsilon})^2e^{2\int_t^T\int_E (1-\rho_s(e)-\epsilon)\phi_s(de)dA_s}\\
&\leq\prod_{t<T_n\leq T}(\rho_{t_n}(\xi_n)+\bar{\epsilon})^2e^{\frac{1}{2}\int_t^T\int_E(1-(\rho_s(e)+\bar{\epsilon})^4)\phi_s(de)dA_s}\\
&\cdot e^{-\frac{1}{2}\int_t^T\int_E(1-(\rho_s(e)+\bar{\epsilon})^4)\phi_s(de)dA_s}e^{2A_T-2A_t}\\
&=\sqrt{\frac{\mathcal{E}((\rho^\mathbf{a}+\bar{\epsilon})^4)_T}{\mathcal{E}((\rho^\mathbf{a}+\bar{\epsilon})^4)_t}}e^{-\frac{1}{2}\int_t^T\int_E(1-(\rho_s(e)+\bar{\epsilon})^4)\phi_s(de)dA_s}e^{2A_T-2A_t}\\
&\leq \sqrt{\frac{\mathcal{E}((\rho^\mathbf{a}+\bar{\epsilon})^4)_T}{\mathcal{E}((\rho^\mathbf{a}+\bar{\epsilon})^4)_t}}e^{(L+\bar{\epsilon})^4(A_T-A_t)-\frac{1}{2}(A_T-A_t)+2(A_T-A_t)}\leq \sqrt{\frac{\mathcal{E}((\rho^\mathbf{a}+\bar{\epsilon})^4)_T}{\mathcal{E}((\rho^\mathbf{a}+\bar{\epsilon})^4)_t}} e^{\frac{\eta}{2}A_T}.
\end{align*}
The last term is integrable (as above), and thus we have found an integrable function that is greater than
$$
\left(\frac{L_T^{\mathbf{a}+\epsilon}}{L_t^{\mathbf{a}+\epsilon}}-\frac{L_T^{\mathbf{a}}}{L_t^{\mathbf{a}}}\right)^2.
$$
Then we can apply the dominated convergence theorem and obtain that
$$
\econdm{\xi^{a_T}+\int_t^Tf_s^{a_s}dA_s-D_T^a}{\mathbf{a}+\epsilon}\rightarrow \econdm{\xi^{a_T}+\int_t^Tf_s^{a_s}dA_s-D_T^a}{\mathbf{a}},
$$
which together with \eqref{swit_representation_epsilon_part} tells us that
\begin{equation*}
Y_t^i\geq J(t,i,\mathbf{a})\quad \prob\text{-a.s.}
\end{equation*}
Using the same computation we can show that
\begin{equation*}
Y_t^i= J(t,i,\mathbf{a^*})\quad \prob\text{-a.s.}
\end{equation*}
and thus $\prob$-a.s.
$$
Y_t^i=v(t,i)
$$
\end{proof}

The representation of $Y^i$ as an essential supremum, together with the fact that it is càdlàg, also tells us that it is unique. Thanks to this we can prove uniqueness of the solution to the system, in a quite straightforward way:
\begin{prop}
	\label{swit_uniqueness}
	The solution to the system \eqref{eq:system} is unique.
\end{prop}

\begin{proof}
	Consider two sets of solutions $(Y^i,U^i,Z^i,K^i)_{i\in\mathcal{J}}$ and $(\bar{Y}^i,\bar{U}^i,\bar{Z}^i,\bar{K}^i)_{i\in\mathcal{J}}$. From theorem \ref{swit:representation_property} we know that for all $i$ $Y^i_t$ and $\bar{Y}^i_t$ both coincide with the value function $v(t,i)$. Since both are càdlàg this means that $Y^i=\bar{Y}^i$ up to indistinguishability. This in turn lets us write, by considering the equations for $Y^i$ and $\bar{Y}^i$ that
	\begin{multline}
	\label{swit_equal_sides_uniqueness}
	\int_t^T\int_E U_s^i(e)(\rho_s^i(e)-1)\phi_s(de)dA_s-\int_t^T\int_E U_s^i(e)q(dsde)-\int_t^TZ_s^idW_s+K_T^i-K_t^i\\
	=\int_t^T\int_E \bar{U}_s^i(e)(\rho_s^i(e)-1)\phi_s(de)dA_s-\int_t^T\int_E \bar{U}_s^i(e)q(dsde)-\int_t^T\bar{Z}_s^idW_s+\bar{K}_T^i-\bar{K}_t^i.
	\end{multline}
	Since both sides must have the same jumps, this means that for all jump $T_n$ with mark $\xi_n$
	$$
	U_{T_n}^i(\xi_n)=\bar{U}^i_{T_n}(\xi_n).
	$$
	This means that $U=\bar{U}$ in $L^{1,0}(p)$, indeed
	\begin{align*}
	\evals{\int_0^T\int_E|U_s(e)-\bar{U}^i_s(e)|\phi_s(de)dA_s}&=\evals{\int_0^T\int_E|U_s(e)-\bar{U}^i_s(e)|p(dsde)}\\
	&=\evals{\sum_{0<T_n\leq T}|U_{T_n}^i(\xi_n)-\bar{U}_{T_n}^i(\xi_n)|}=0.
	\end{align*}
	Now \eqref{swit_equal_sides_uniqueness} becomes
	$$
	\int_t^T(Z_s^i-\bar{Z}_s^i)dW_s=K_T^i-K_t^i-\bar{K}_T^i+\bar{K}_t^i.
	$$
	Consider it between $0$ and $T$. By taking the predictable bracket against $\int_0^T(Z_s^i-\bar{Z}_s^i)dW_s$ on both sides we obtain that, since the $K^i$ and $\bar{K}^i$ are finite variation processes,
	$$
	\int_0^T (Z_s^i-\bar{Z}_s^i)^2dW_s=0,
	$$
	which tells us that also the $Z$ component is unique.
	This leaves us with only $K$:
	$$
	K_T^i-K_t^i-\bar{K}_T^i+\bar{K}_t^i=0.
	$$
	It is easy to see that by setting $t=0$ we obtain $K_T^i=\bar{K}_T^i$ and from that for any $t\in[0,T]$ $K_t^i=\bar{K}_t^i$.
\end{proof}

\appendix
\section{A comparison theorem for BSDE driven by MPP and Brownian motion}
\label{appendix_comparison}
In this section we establish a comparison theorem for BSDE driven by marked point process and Brownian motion. Consider the following BSDE
\begin{equation}
\label{appendix_BSDE_equation_with_W}
Y_t=\xi+\int_t^Tf_s(Y_s,U_s)dA_s+\int_t^Tg_s(Y_s,Z_s)ds-\int_t^T\int_E U_s(e)q(dsde)-\int_t^TZ_sdW_s.
\end{equation}

The solution is a triple $(Y,U,Z)\in L^{2,\beta}(A)\cap L^{2,\beta}(W)\times L^{2,\beta}(p)\times L^{2,\beta}(W)$. It is possible to prove, in a quite standard way, that a solution exists under the following hypotheses
\begin{Asmpt}{(R)}\ \\
	\label{ass:appendix_BSDE}
	\begin{enumerate}
		\item The final condition $\xi:\Omega\rightarrow \mathbb{R}$ is $\mathcal{G}_T$ measurable and $\evals{e^{\beta A_T}|\xi|^2}<\infty$.
		\item For every $\omega\in\Omega$, $t\in[0,T]$, $r\in\mathbb{R}$, the mapping $f_t(\omega,r,\cdot):\mathcal{L}^2(E,\mathcal{E},\phi_t(\omega,de))\rightarrow\mathbb{R}$ satisfies the assumptions:
		\begin{enumerate}[label=(\roman*)]
			\item for every $U\in\mathcal{L}^{2,\beta}(p)$ the mapping
			$$
			(\omega,t,r)\mapsto f_t(\omega,r,U_t(\omega,\cdot))
			$$
			is $Prog^{\mathbb{G}}\otimes\mathcal{B}(\mathbb{R})$-measurable;
			\item there exists $L_f\geq0$ and $L_U\geq 0$ such that for every $\omega\in\Omega$, $t\in[0,T]$, $r,r'\in\mathbb{R}$, $z,z'\in \mathcal{L}^2(E,\mathcal{E},\phi_t(\omega,de))$ we have
			\begin{multline*}
			|f_t(\omega,r,z(\cdot))-f_t(\omega,r',z'(\cdot))|\\
			\leq L_f|r-r'|+L_U\left(\int_E|z(e)-z'(e)|^2\phi_t(\omega,de)\right)^2;
			\end{multline*}
			\item we have
			$$
			\evals{\int_0^Te^{\beta A_t}|f_t(0,0)|^2dA_t}<\infty.
			$$
		\end{enumerate}
	\item 	\begin{enumerate}[label=(\roman*)]
		\item $g $ is $Prog\times\mathcal{B}(\mathbb{R})\times\mathcal{B}(\mathbb{R}^d)$-measurable.
		\item There exist $L_g\geq 0$, $L_Z\geq 0$ such that for every $\omega\in\Omega$,
		$t\in\left[0,T\right]$, $y,y'\in\mathbb{R}$, $z,z'\in\mathbb{R}^d$
		$$
		|g(\omega,t,y,z)-g(\omega,t,y',z')|\leq L_g|y-y'|+L_Z|z-z'|
		$$
		\item we have
		$$
		\eval\left[\int_0^Te^{\beta A_s}|g(s,0,0)|^2ds\right]<\infty.
		$$
	\end{enumerate}
	\end{enumerate}
\end{Asmpt}
In that case we have the result
\begin{teo}
	\label{appendix_simple_BSDE}
	Suppose that assumption \ref{ass:appendix_BSDE} hold for some $\beta>2L_f+L_U^2$. Then there exists a unique solution to equation \eqref{appendix_BSDE_equation_with_W} $(L^{2,\beta}(A)\cap L^{2,\beta}(W))\times L^{2,\beta}(p)\times L^{2,\beta}(W)$.
\end{teo}
We do not give the proof here as it is quite standard and a straightforward extension of the one in \cite{confortola2013backward}.

We want to prove a comparison theorem for this equation. It is well known that BSDE with jumps component require additional assumptions for this result to hold (see \cite{barles1997backward} for a counterexample), and BSDE driven marked point processes are no exception. There are a number of works that provide a comparison principle for BSDE with jumps, and we cite among others \cite{cohen10generalcomparison,kruse2016monotonecomparison,quenez2013jumpsoptimizationdynamicrisk,royer2006relatednonlinear}. We give here a result for the BSDE \eqref{appendix_BSDE_equation_with_W}. Compared to the case of point processes that have a compensator absolutely continuous with respect to the Lebesgue measure, we need to add some integrability condition on the process $A$ since we will be using Girsanov changes of measure.

\begin{teo}
	\label{appendix_BSDE_teo_general_comparison_simple_BSDE}
	Let $(\xi^i,f^i,g^i)_{i=1,2}$ be two sets of data for which the hypotheses above hold. Let $(Y^i,U^i,Z^i)_{i=1,2}$ be the corresponding solutions. Assume that $\xi^2\leq\xi^1$ a.s, $f_t^2(Y_t^1,U_t^1)\leq f_t^1(Y_t^1,U_t^1)$ and $g_t^2(Y_t^1,Z_t^1)\leq g_t^1(Y_t^1,Z_t^1)$ a.s. for all $t$. Assume there is a $\mathcal{P}\otimes\mathcal{E}$-measurable function $\gamma$ such that $-1\leq\gamma \leq C$ and for all $t,y\in\mathbb{R}^+\times\mathbb{R}$ and $u^2,u^1\in \mathcal{L}^2(E,\mathcal{E},\phi_t(\omega,de))$
	\begin{equation}
	\label{appendix_BSDE_gamma_condition}
	f_t^2(y,u^2)-f_t^2(y,u^1)\leq\int_E\gamma_t(e)(u^2(e)-u^1(e))\phi_t(de).
	\end{equation}
	Assume also that for some $\eta>3+(C+1)^4$ we have $\evals{e^{\eta A_T}}<\infty$.
	Then we have that $Y_t^2\leq Y_t^1$ a.s. for all $t$.
\end{teo}

\begin{proof}
	To simplify notation suppose that the Brownian motion is in $d=1$, the general case is done as in the case with only a diffusion part. Define $\bar{Y}=Y^2-Y^1$, $\bar{U}=U^2-U^1$, $\bar{Z}=Z^2-Z^1$, $\bar{\xi}=\xi^2-\xi^1$, $ \bar{f}=f^2(Y^1,U^1)-f^1(Y^1,U^1) $ and $ \bar{g}=g^2(Y^1,Z^1)-g^1(Y^1,Z^1) $.
	$\bar{Y}$ satisfies
	\begin{multline}
	\bar{Y}_t=\bar{\xi}+\int_t^T(f_s^2(Y_s^2,U_s^2)-f_s^1(Y_s^1,U_s^1))dA_s+\int_t^T(g_s^2(Y_s^2,Z_s^2)-g_s^1(Y_s^1,Z_s^1))dA_s\\-\int_t^T\int_E \bar{U}_s(e)q(dsde)-\int_t^T\bar{Z}_sdW_s
	\end{multline}
	Define the quantities
	\begin{gather*}
	\alpha_s=\frac{f_s^2(Y_s^2,U_s^2)-f_s^2(Y_s^1,U_s^2)}{\bar{Y}_s}\indb{\bar{Y}_s\neq 0}\leq L_f\\
	\beta_s=\frac{g_s^2(Y_s^2,Z_s^2)-g_s^2(Y_s^1,Z_s^2)}{\bar{Y}_s}\indb{\bar{Y}_s\neq 0}\leq L_g\\
	\theta_s=\frac{g_s^2(Y_s^1,Z_s^2)-g_s^2(Y_s^1,Z_s^1)}{\bar{Z}_s}\indb{\bar{Z}_s\neq 0}\leq L_Z.
	\end{gather*}
	The equation can be thus rewritten as
	\begin{multline*}
	\bar{Y}_t=\bar{\xi}+\int_t^T\alpha_s\bar{Y}_sdA_s+\int_t^T(f_s^2(Y_s^1,U_s^2)-f_s^2(Y_s^1,U_s^1))dA_s+\int_t^T\bar{f}_sdA_s\\+\int_t^T\beta_s\bar{Y}_sds+\int_t^T\theta_s\bar{Z}_sds+\int_t^T\bar{g}_sds -\int_t^T\int_E \bar{U}_s(e)q(dsde)-\int_t^T\bar{Z}_sdW_s.
	\end{multline*}
	Consider now the positive process $\Gamma$ solution to
	\begin{equation}
	\begin{cases}
	d\Gamma_t=\alpha_tdA_t+\beta_tdt\\
	\Gamma_0=1.
	\end{cases}
	\end{equation}
	We have that
	
	$$
	\Gamma_t=e^{\int_0^t\alpha_sdA_s}e^{\int_0^t\beta_sds}=B_tC_t,
	$$
		where $B_t=\exp\left\lbrace\int_0^t\alpha_sdA_s\right\rbrace$ and $C_t=\exp\left\lbrace\int_0^t\beta_sds\right\rbrace<e^{L_gT}$.
	We notice that
	\begin{equation}
	\label{mathintro_B_is_smaller}
	B_t^2\leq e^{2L_f A_t}<e^{\beta A_t}.
	\end{equation}
	We consider now the dynamic of the product $\bar{Y}\Gamma$. We obtain
	\begin{multline*}
	\Gamma_T\bar{\xi}=\Gamma_t\bar{Y}_t+\int_t^T\alpha_s\bar{Y}_s\Gamma_sdA_s+\int_t^T\beta_s\bar{Y}_s\Gamma_sds-\int_t^T\alpha_s\bar{Y}_s\Gamma_sdA_s\\
	-\int_t^T\Gamma_s(f_s^2(Y_s^1,U_s^2)-f_s^2(Y_s^1,U_s^1))dA_s-\int_t^T\Gamma_s\bar{f}_sdA_s\\
	-\int_t^T\beta_s\bar{Y}_s\Gamma_sds-\int_t^T\theta_s\bar{Z}_s\Gamma_s ds-\int_t^T\Gamma_s\bar{g}_sds\\
	+\int_t^T\int_E \Gamma_s\bar{U}_s(e)q(dsde)+\int_t^T\Gamma_s\bar{Z}_sdW_s.
	\end{multline*}
	Remembering that $\bar{\xi}$, $\bar{f}$ and $\bar{g}$ are non-positive while $\Gamma$ is non-negative we obtain

	\begin{multline*}
	\bar{Y}_t\Gamma_t\leq \int_t^T\Gamma_s(f_s^2(Y_s^1,U_s^2)-f_s^2(Y_s^1,U_s^1))dA_s+\int_t^T\theta_s\bar{Z}_s\Gamma_s ds\\
	-\int_t^T\int_E \Gamma_s\bar{U}_s(e)q(dsde)-\int_t^T\Gamma_s\bar{Z}_sdW_s.
	\end{multline*}

	Using the condition \eqref{appendix_BSDE_gamma_condition} the last inequality becomes
	
	\begin{multline}
	\label{mathintro_inequality_before_reordering}
	\bar{Y}_t\Gamma_t\leq \int_t^T\int_E\gamma_s(e)\Gamma_s\bar{U}_s(e)\phi_s(e)dA_s+\int_t^T\theta_s\bar{Z}_s\Gamma_s ds\\
	-\int_t^T\int_E \Gamma_s\bar{U}_s(e)q(dsde)
	-\int_t^T\Gamma_s\bar{Z}_sdW_s.
	\end{multline}
	We have that $\Gamma\bar{U}\in L^{2,0}(p)$ and $\Gamma\bar{Z}\in L^{2,0}(W)$ and the terms in $q(dsde)$ and $dW$ are martingales, indeed
	\begin{gather*}
	\evals{\int_0^T\int_E|\Gamma_s\bar{U}_s(e)|^2\phi_s(de)dA_s}<\evals{\int_0^T\int_Ee^{\beta A_s}|\bar{U}_s|^2(e)\phi_s(de)dA_s}<\infty\\
	\evals{\int_0^T|\Gamma_s\bar{Z}_s|^2ds}<\evals{\int_0^Te^{\beta A_s}|\bar{Z}_s|^2ds}<\infty\\
	\end{gather*}
	Now reorder the terms in \eqref{mathintro_inequality_before_reordering} to obtain
	\begin{multline*}
	\bar{Y}_t\Gamma_t\leq -\int_t^T\Gamma_s\bar{Z}_s(dW_s-\theta_sds)\\
	-\int_t^T\int_E \Gamma_s\bar{U}_s(e)(p(dsde)-(\gamma_s(e)+1)\phi_s(de)dA_s.
	\end{multline*}
	Consider $\bar{\epsilon}$ such that $\eta>3+(C+1+\bar{\epsilon})^4$. Then for $\epsilon<\bar{\epsilon}$ we add to both sides of the previous inequality the term 
	$$
	\epsilon\int_t^T\Gamma_s\bar{U}_s(e)\phi_s(de)dA_s,
	$$
	obtaining
	\begin{multline}
	\label{mathintro_befor_epsilon_exp}
	\bar{Y}_t\Gamma_t+\epsilon\int_t^T\Gamma_s\bar{U}_s(e)\phi_s(de)dA_s\leq  -\int_t^T\Gamma_s\bar{Z}_s(dW_s-\theta_sds)\\
	-\int_t^T\int_E \Gamma_s\bar{U}_s(e)(p(dsde)-(\gamma_s(e)+1+\epsilon)\phi_s(de)dA_s.
	\end{multline}
	Now we can consider $\gamma_s(e)+1+\epsilon$ (which satisfies, together with the condition on $A$ in the statement of the theorem, assumption \ref{ass:rhosbound}) and $\theta_s$ (which is bounded) as Girsanov kernels. We can introduce the probability $\prob^{\gamma,\theta,\epsilon}\sim\prob$ through the exponential martingale here denoted $L^{\gamma,\theta,\epsilon}$. It is equivalent since the part relative to the point process of Girsanov kernel is strictly positive.
	All $\prob$-martingales are $\prob^{\gamma,\theta,\epsilon}$ martingales and thus by taking $\prob^{\gamma,\theta,\epsilon}$ expectation conditional on $\mathcal{F}_t$ in \eqref{mathintro_befor_epsilon_exp} they vanish. Thus for any $\epsilon<\bar{\epsilon}$ we have that
	$$
	\bar{Y}_t\Gamma_t+\epsilon\econdm{\int_t^T\Gamma_s\bar{U}_s(e)\phi_s(de)dA_s}{\gamma,\theta,\epsilon}\leq 0 \quad \prob^{\gamma,\theta,\epsilon}\text{-a.s.}
	$$
	and thus
	$$
	\bar{Y}_t\Gamma_t+\epsilon\econdm{\int_t^T\Gamma_s\bar{U}_s(e)\phi_s(de)dA_s}{\gamma,\theta,\epsilon}\leq 0 \quad \prob\text{-a.s.}
	$$
	since they are equivalent. To conclude the theorem we just have to show that
	\begin{equation}
	\label{mathintro_convergence_of_added_term}
	\epsilon\econdm{\int_t^T\Gamma_s\bar{U}_s(e)\phi_s(de)dA_s}{\gamma,\theta,\epsilon}\rightarrow 0.
	\end{equation}
	This can be done as in the proof of proposition \ref{swit_existence_picard_iteration}.

	Then \eqref{mathintro_convergence_of_added_term} holds and the theorem is proven.
\end{proof}

\section{A monotonic limit result}
We establish a monotonic limit proposition, as the one introduced in \cite{peng1999monotonic}, for BSDEs of a very particular type. The generator is linear in $U$, and it is the kind of equation we need in this paper.
Consider the following sequence of BSDE
\begin{multline}
\label{mathintro_monotonic_limit_BSDEs}
Y_t^n=\xi+\int_t^Tf_sdA_s+\int_t^T\int_EU_s^n(e)(\rho_s(e)-1)\phi_s(de)dA_s+\int_t^Tg_sds\\-\int_t^T\int_EU_s^nq(dsde)-\int_t^TZ_s^ndW_s+K_T^n-K_t^n
\end{multline}
$K^n$ are known non-decreasing predictable processes starting from zero that are square integrable, that is $\evals{ (K^n_T)^2 } <\infty$. In the following, $\rho$ is a $\mathcal{P}\otimes\mathcal{E}$-measurable random field such that $0\leq \rho_s(e)\leq M$ for some constant $M$, and the hypothesis on $\xi$, $f$ and $g$ are the ones already listed in appendix \ref{appendix_comparison} for  $\beta>(\max\lbrace|M-1|,1\rbrace)^2$. 

\begin{prop}
	\label{prop:monotonic_point}
	Assume that for $\beta>(\max\lbrace|L-1|,1\rbrace)^2$,  $Y_t^n\nearrow Y_t$ for some $Y_t\in L^{2,\beta}(A)\cap L^{2,\beta}(W)$ and
	$$
	\|U^{n}\|_{L^{2,\beta}(p)}+\|Z^{n}\|_{L^{2,\beta}(W)}\leq C.
	$$
	Assume $K^n$ are non decreasing càdlàg predictable with $K_0^n=0$ and $\evals{(K^n_T)^2}<\infty$. Then there exists $U\in L^{2,\beta}(p)$, $Z\in L^{2,\beta}(W)$ and $K$ predictable càdlàg increasing process with $\eval[K_T^2]<\infty$, such that
	\begin{multline*}
	Y_t=\xi+\int_t^Tf_sdA_s+\int_t^T\int_EU_s(e)(\rho_s(e)-1)\phi_s(de)dA_s+\int_t^Tg_sds\\-\int_t^T\int_EU_s(e)q(dsde)-\int_t^TZ_sdW_s+K_T-K_t.
	\end{multline*}
\end{prop}

\begin{proof}
	First notice that $Y$ is also the $L^{2,\beta}([0,T])$ limit of the $Y^n$. Thanks to the bound on $U^n$ and $Z^n$, we can extract a weakly convergent subsequence $(U^{n_k},Z^{n_k})$ to some $(U,Z)$. This implies that for every fixed stopping time $\tau$,
	\begin{gather*}
	\int_0^\tau\int_EU_s^{n}(e)q(dsde)\overset{w}{\rightharpoonup}\int_0^\tau\int_EU_s(e)q(dsde)\\
	\int_0^\tau Z_s^{n}dW_s\overset{w}{\rightharpoonup}\int_0^\tau Z_s dW_s
	\end{gather*}
	since the application that associates $\int_0^\tau\int_E V_s(e)q(dsde) \in L^2(\mathcal{F}_\tau)$ to $V\in L^{2,\beta}(p)$ is continuous, and the same for the application that associates $\int_0^\tau X_sdW_s \in  L^2(\mathcal{F}_\tau)$ to $X\in L^{2,\beta}(W)$ . Notice also that the linear application that to $V\in L^{2,\beta}(p)$ associates $\int_0^\tau\int_E V_s(e)(\rho_s(e)-1)\phi_s(de)dA_s\in L^2(\mathcal{F}_\tau)$ is continuous.
	
	If we define $K_\tau$ as
	\begin{multline*}
	K_\tau= Y_0-Y_\tau-\int_0^\tau f_sdA_s-\int_{0}^{\tau}g_sds-\int_0^\tau\int_E U_s(e)(\rho_s(e)-1)\phi_s(de)dA_s\\+\int_0^\tau\int_EU_s(e)q(dsde)+\int_0^\tau Z_sdW_s,
	\end{multline*}
	since $K^n_\tau$ satisfies
	\begin{multline*}
	K^n_\tau= Y^n_0-Y^n_\tau-\int_0^\tau f_sdA_s-\int_0^\tau g_s ds-\int_0^t\int_E U^n_s(e)(\rho_s(e)-1)\phi_s(de)dA_s\\+\int_0^\tau\int_EU^n_s(e)q(dsde)+\int_0^\tau Z_s^ndW_s
	\end{multline*}
	it follows that $K^{n_k}_\tau\overset{w}{\rightharpoonup}K_\tau$, for all $\tau$.
	We can deduce that $K_\tau$ inherits the following properties from $K_\tau^n$:
	\begin{itemize}
		\item $K_0=0$ and $\eval[K_T^2]<\infty$.
		\item $K$ is increasing.
	\end{itemize}
	Thanks to Lemma 2.2 in \cite{peng1999monotonic}, we also know that both $Y$ and $K$ are càdlàg.
\end{proof}

\section{Approximation and comparison of the specific reflected BSDE}

In this appendix we establish a comparison result for reflected BSDE of the form \eqref{app_other_BSDE_equation_to_approx}. We first show we can approximate the reflected BSDE with a sequence of penalized standard BSDE, on which we then apply the comparison theorem \ref{appendix_BSDE_teo_general_comparison_simple_BSDE}. It is well known that general comparison theorems do not hold in this framework with jumps, both for normal BSDE and reflected ones. Nevertheless, with additional hypotheses as in theorem \ref{appendix_BSDE_teo_general_comparison_simple_BSDE}, it is possible to compare them. Some examples of papers in which a comparison result for reflected BSDE with jumps is established is \cite{crepey08refelctedcomparison,essaky2008rbsdepenalization,ren10,quenez14robustoptimal}, along many others.

The equation we study is again
\begin{equation}
\label{app_other_BSDE_equation_to_approx}
\begin{cases}
Y_t=\xi+\int_t^Tf_sdA_s+\int_t^TU_s(e)(\rho_s(e)-1)\phi_s(de)dA_s+\int_t^Tg_sds\\
\qquad-\int_t^T U_s(e)q(dsde)-\int_t^TZ_sdW_s+K_T-K_t\\
Y_t\geq h_t\\
\int_0^T(Y_{t^-} -h_{t^-})dK_t=0.
\end{cases}
\end{equation} 
We assume that $\rho$ and $A$ satisfy assumption \ref{ass:rhosbound} for appropriate $M>0$ and $\eta>3+M^4$. For $\beta > (M')^2$ (where $M'=\max(M,|M-1|)$, we assume that data $\xi$ is a $\mathcal{F}_T$-measurable random variable such that $\evals{e^{\beta A_T}|\xi^2}<\infty$, $f,\;g$ to be progressive processes in $L^{2,\beta}(A)$ and $L^{2,\beta}(W)$ respectively and $h$ to be a càdlàg adapted process such that $\evals{\sup_{t\in[0,T]}e^{(\beta+\delta) A_t|h_t|^2}}<\infty$ for some $\delta>0$. Then this equation has one unique solution in $ \left(L^{2,\beta}(A)\cap L^{2,\beta}(W)\right)\times L^{2,\beta}(p)\times L^{2\beta}(W)\times \mathcal{I}^2$ for $\beta>(M')^2$ thanks to \cite[Theorem 4.1]{foresta2017optimal}.

We start by defining for all $n\geq 0$ the following penalized BSDE

\begin{multline}
\label{otherBSDE_penalized}
Y^n_t=\xi+\int_t^Tf_sdA_s+\int_t^T\int_EU_s^n(e)(\rho_s(e)-1)\phi_s(de)dA_s+\int_t^Tg_sds\\
-\int_t^TU^n_s(e)q(dsde)-\int_t^TZ_s^ndW_s+K_T^n-K_t^n,
\end{multline}
where $K_t^n=n\int_0^t\left(Y_s^n-h_s\right)^-ds$.
We see that for each $n$ this is a standard BSDE with generators $\bar{f}^n_s(u)=f_s+\int_E(e)u(e)(\rho_s(e)-1)\phi_s(de)$ and $\bar{g}^n_s(y)=g_s+n(y-h_s)^-$, thus existence and uniqueness in $\left(L^{2,\beta}(A)\cap L^{2,\beta}(W)\right)\times L^{2,\beta}(p)\times L^{2\beta}(W)$ for $\beta>(L')^2$ is assured by theorem \ref{appendix_simple_BSDE}, where $L'=\max\lbrace|L-1|,1\rbrace$. The comparison theorem \ref{appendix_BSDE_teo_general_comparison_simple_BSDE} tells us that $Y_t^n\leq Y_t^{n+1}$ since all the hypotheses are verified (in particular $\bar{g}^n(y)\leq\bar{g}^{n+1}(y)$ for all $y\in\mathbb{R}$. We want to prove we can approximate \eqref{app_other_BSDE_equation_to_approx} with \eqref{otherBSDE_penalized}. Let us start by some lemmas

\begin{lemma}
	\label{otherBSDE_lemma_hatYsmaller}
	Let $Y^n$ be the solution to \eqref{otherBSDE_penalized}. Then there exists a $\hat{Y}$ such that $Y^n_t\nearrow \hat{Y}_t\leq Y_t$, where $Y$ is the solution to the RBSDE \eqref{app_other_BSDE_equation_to_approx}.
\end{lemma}

\begin{proof}
	Consider $\bar{\epsilon}>0$ such that $\eta>3+(L'+\epsilon)^4$, where $\eta$ is the parameter appearing in assumption \ref{ass:rhosbound}. Then $\rho+\epsilon$ induces an equivalent probability $\prob^{\rho+\epsilon}\sim\prob$. Then it is possible to show that for any $0<\epsilon<\bar{\epsilon}$ $Y$ satisfies
	
	\begin{equation}
	\label{otherBSDE_snell_rho}
	Y_t+\epsilon \Gamma^U=\essup_{\tau\geq t}\econdm{\int_t^{\tau\wedge T}f_sdA_s+\int_t^{\tau\wedge T}g_sds+h_\tau\indb{\tau<T}+\xi\indb{\tau\geq T}}{\rho+\epsilon},
	\end{equation}
	where
	$$
	\Gamma^U=\essup_{\tau\geq t}\econdm{\int_t^\tau\int_E U_s(e)\phi_s(de)dA_s}{\rho+\epsilon}.
	$$
	
	As already said, we have that $Y^n\leq Y^{n+1}$ thanks to the comparison theorem. Now for fixed $n$ consider $Y^n$ between t and a generic stopping time $\tau$, we have
	\begin{align}
	Y_t^n+
	&\epsilon\int_t^\tau\int_E U_s^n(e)\phi_s(de)dA_s\\
	&\geq\econdm{Y_\tau^n\wedge h_\tau\indb{\tau<T}+\xi\indb{\tau\geq T}+\int_t^\tau f_sdA_s+\int_t^{\tau}g_sds}{\rho+\epsilon}.
	\label{otherBSDE_snell_penalized}
	\end{align}
	Define now $\tau^*_t=\inf\lbrace s \geq t : K_s^n-K_t^n>0\rbrace\wedge T$.
	Take $\omega$ such that $\tau^*_t(\omega)<T$. Then $\exists t_k\searrow\tau^*_t(\omega)$ such that $Y_{t_k}^n(\omega)\leq h_{t_k}(\omega)$. Then $Y_{\tau^*_t}^n(\omega)\leq h_{\tau^*_t}(\omega)$, since $Y^n$ and $h$ are càdlàg. Then
	$$
	Y_{\tau^*_t}^n\indb{\tau_t^*<T}=Y_{\tau^*_t}^n\wedge h_{\tau^*_t}\indb{\tau_t^*<T}
	$$
	and
	\begin{multline}
	Y_t^n+\epsilon\int_t^{\tau^*_t}\int_EU_s^n(e)\phi_s(de)dA_s=Y_{\tau^*_t}^n\wedge h_{\tau^*_t}\indb{\tau_t^*<T}+\xi\indb{\tau^*_t\geq T}+\int_t^{\tau^*_t}g_sds\\
	+\int_t^{\tau^*_t}f_sdA_s+\int_t^{\tau^*_t}\int_EU_s^n(e)(\rho_s(e)+\epsilon-1)\phi_s(de)dA_s\\-\int_t^{\tau^*_t}\int_EU_s^n(e)q(dsde)-\int_t^{\tau^*_t}Z_s^ndW_s.
	\end{multline}
	By taking $\rho-$expectation conditional on $\mathcal{F}_t$ we obtain, together with \eqref{otherBSDE_snell_penalized}
	\begin{multline*}
	Y_t^n+\epsilon\Gamma^{U^n}\\=\essup_{\tau\geq T}\econdm{Y_\tau^n\wedge h_\tau\indb{\tau<T}+\xi\indb{\tau\geq T}+\int_t^{\tau\wedge T} f_sdA_s+\int_t^{\tau\wedge T} g_sds}{\rho+\epsilon}.
	\end{multline*}
	Where
	$$
	\Gamma^{U^n}=\essup_{\tau\geq t}\econdm{\int_t^\tau\int_E U_s(e)\phi_s(de)dA_s}{\rho+\epsilon}
	$$
	Comparing this last to \eqref{otherBSDE_snell_rho} we notice that $\prob^{\rho+\epsilon}$-a.s $Y^n_t+\epsilon\Gamma^{U^n}\leq Y_t+\epsilon\Gamma^U$. Since $\Gamma^{U^n}$ is non-negative, this also means $Y^n_t\leq Y_t+\epsilon\Gamma^U$. Since the sequence $Y^n$ is non-decreasing we have the existence of a limit $\hat{Y}$ such that
	$$
	Y_t^n\nearrow \hat{Y}_t\leq Y_t+\epsilon\Gamma^U.
	$$
	The last inequality holds $\prob^{\rho+\epsilon}$ and $\prob$ almost surely as they are equivalent.
	Now we just have to show that $\Gamma^U$ is bounded so we can send $\epsilon$ to zero and prove the lemma. This is done as in the proof of proposition \ref{swit_existence_picard_iteration}.
\end{proof}

We establish here a lemma on the norms of the solutions to the penalized equations:

\begin{lemma}
	\label{otherBSDE_lemma_penalized_bound}
	Let $(Y^n,U^n,Z^n)$ be the solution to the penalized BSDE \eqref{otherBSDE_penalized}. Then there exist a constant $C_p$ depending on $\xi,f,g,h$ but independent of $n$ such that
	\begin{multline*}
	\evals{\int_0^Te^{\beta A_s}|Y_s^n|^2(dA_s+ds)}+\evals{\int_0^T\int_E e^{\beta A_s}|U_s^n(e)|^2\phi_s(de)dA_s}\\+\evals{\int_0^Te^{\beta A_s}Z_s^2ds}<C_p.
	\end{multline*}
\end{lemma}

\begin{proof}
	The bound is obtained by applying It\^o's formula to $e^{\beta (A_t+t)}(Y_t^n)^2$, as in the proof of \cite[Proposition 3.1]{foresta2017optimal} In the computation appears the term
	$$
	\int_0^Te^{\beta (A_s+s)}Y_s^ndK^n_s=\int_{0}^{T}e^{\beta (A_s+s)}Y_s^n n (Y_s^n-h_s)^-ds\leq \int_0^Te^{\beta (A_s+s)}h_sdK_s^n,
	$$
	where the last inequality is due to the fact that the integrand is non zero only when $h_s\geq Y_s^n$. This is in turn bounded by
	$$
	+\gamma\evals{\sup_t e^{(\beta+\delta)(A_t+t)}h_{t^-}^2}
	+\frac{1}{\gamma}\evals{\left(\int_0^{T}e^{(\beta-\delta)\frac{A_s+s}{2}}dK^n_s\right)^2}
	$$
	for any $\gamma>0$.

	The last term is estimated again by considering the dynamic of $e^{(\beta-\delta)\frac{(A_t+t)}{2}}Y_t^n$.
\end{proof}

\begin{prop}
	\label{otherBSDE_approximation_final}
	We have that $\hat{Y}_t=Y_t$ and thus $Y_t^n\nearrow Y_t$.
\end{prop}

\begin{proof}
	First notice that by Lebesgue dominated convergence theorems we have that $Y^n\overset{L^{2,\beta}(A)}{\rightarrow}\hat{Y}$. Thanks to lemma \ref{otherBSDE_lemma_penalized_bound}, we know that $(U^n,Z^n)$ are bounded in $L^{2,\beta}(p)\times L^{2,\beta}(W)$ for $\beta >(L')^2$. Then thanks to proposition \ref{prop:monotonic_point} we know that there exist $(\hat{U},\hat{Z},\hat{K})$ that solve
	\begin{multline}
	\label{otherBSDE_eq_for_hatY}
	\hat{Y}_t=\xi+\int_t^Tf_sdA_s+\int_t^Tg_sds+\int_t^T\int_E \hat{U}_s(e)(\rho_s(e)-1)\phi_s(de)dA_s\\
	-\int_t^T\int_E \hat{U}_s(e)q(dsde)-\int_t^T \hat{Z}_sdW_s+\hat{K}_T-\hat{K}_t.
	\end{multline}
	Now it holds that $\hat{Y}_t\geq h_t$. Indeed if we take \eqref{otherBSDE_penalized} between $0$ and $T$ and we take expectation:
	\begin{multline*}
	n\eval\left[\int_0^T(Y_s^n-h_s)^-ds\right]=\eval[Y_0^n]\\-\eval\left[\xi-\int_0^T\left(f_s\int_EU-s^n(e)(\rho_s(e)-1)\phi_s(de)\right)dA_s-\int_0^tg_sds\right],
	\end{multline*}
	and, dividing by $n$, $\eval\left[\int_0^T(Y_s^n-h_s)^-ds\right]\rightarrow0$, where we used the fact that $\eval\int_0^T\int_E U_s^n(e)(\rho_s(e)-1)\phi_s(de)dA_s\leq M$ from lemma \ref{otherBSDE_lemma_penalized_bound}.
	Then
	\begin{multline*}\eval\left[\int_0^T(\hat{Y}_s-h_s)^-ds\right]=\eval\left[\int_0^T\lim_n(Y_s^n-h_s)^-ds\right]\\\leq\liminf_n\eval\left[\int_0^T(Y_s^n-h_s)^-ds\right]=0.\end{multline*}
	Then $\mathbb{P}$-a.s. $\int_0^T(Y_s^n-h_s)^-ds=0$ and since the process are càdlàg, $\mathbb{P}$-a.s. $\hat{Y}_s\geq h_s$ for all $t<T$. Since $\hat{Y}_T=\xi$, this holds also for $T$.
	This means that
	
	\begin{equation}
	\label{otherBSDE_hatY_larger}
	\hat{Y}_t+\int_0^tf_sdA_s+\int_0^tg_sds\geq \xi\indb{t\geq T}+h_t\indb{t<T}+\int_0^tf_sdA_s+\int_0^tg_sds.
	\end{equation}
	
	Introduce again (see the proof of lemma \ref{otherBSDE_lemma_hatYsmaller}) the equivalent probability $\prob^{\rho+\epsilon}\sim\prob$ through a Girsanov transfomr with kernel $\rho+\epsilon$.
	Define the non-negative quantity
	$$
	\hat{\Gamma}^U_t=\essup_{\tau\geq t}\econdm{\int_t^\tau\int_E \hat{U}_s(e)\phi_s(de)dA_s}{\rho+\epsilon}.
	$$
	If we consider equation \eqref{otherBSDE_eq_for_hatY}  between t and a stopping time $\tau$ and add $$\epsilon\int_t^\tau\int_E \hat{U}_s(e)\phi_s(de)dA_s$$ to both sides, we obtain after taking $\essup$ over all stopping times on both sides that
	
	\begin{equation*}
	\hat{Y}_t+\int_0^tf_sdA_s+\int_0^tg_sds+\epsilon\hat{\Gamma}^U_t=\essup_{\tau\geq t}{\hat{\psi}_\tau}-\hat{K}_t,
	\end{equation*}
	where
	$$
	\hat{\psi_t}=\xi+\int_0^tf_sdA_s+\int_0^tg_sds+\hat{K}_t.
	$$
	Thus
	\begin{equation}
	\label{otherBSDE_second_supermg}
	\hat{Y}_t+\int_0^tf_sdA_s+\int_0^tg_sds+\epsilon\hat{\Gamma}^U_t
	\end{equation}
	is a $\prob^{\rho+\epsilon}$-supermartingale as it is the difference of a supermartingale (it is a Snell envelope) and an increasing process. Notice that since $\prob^{\rho+\epsilon}\sim\prob$, the inequality \eqref{otherBSDE_hatY_larger} also holds $\prob^{\rho+\epsilon}$-a.s. Also if we add $\hat{\Gamma}^U_t$ to the left hand side of \eqref{otherBSDE_hatY_larger}, the inequality is still true. Then \eqref{otherBSDE_second_supermg} is a $\prob^{\rho+\epsilon}$-supermartingale that dominates
	$$
	\psi_t=\xi\indb{t\geq T}+h_t\indb{t<T}+\int_0^tf_sdA_s+\int_0^tg_sds.
	$$
	On the other hand we have that (see the proof of lemma \ref{otherBSDE_lemma_hatYsmaller})
	$$
	Y_t+\epsilon \Gamma^U=Y_t+\epsilon\essup_{\tau\geq t}\econdm{\int_t^\tau\int_E U_s(e)\phi_s(de)dA_s}{\rho+\epsilon}
	$$
	is the $\prob^{\rho+\epsilon}$ Snell envelope of the same quantity $\psi_t$. As the Snell envelope is the smallest supermartingale that dominates $\eta_t$, it holds that $\prob^{\rho+\epsilon}$-a.s.
	$$
	Y_t+\epsilon\Gamma^U\leq \hat{Y}+\epsilon\hat{\Gamma}^U.
	$$
	But since the two probabilities are equivalent, this means that the inequality holds also $\prob$-a.s. Now we have already shown in the proof of lemma \ref{otherBSDE_lemma_hatYsmaller} that
	$$
	\Gamma^U_t\leq \econd{e^{\eta A_t}}+\frac{\econd{\int_0^T\int_Ee^{\beta A_s}U_s^2(e)\phi_s{de}dA_s}}{4\beta}.
	$$
	The same holds for $ \hat{\Gamma}^U_t $. Thus by sending $\epsilon$ to zero we obtain that $Y_t\leq\hat{Y}_t$ $\prob$-a.s. Since the processes are càdlàg this holds up to indistinguishability. Confronting it with lemma \ref{otherBSDE_lemma_hatYsmaller} we obtain that $Y_t=\hat{Y}_t$.
\end{proof}

Thanks to the approximation result, we can provide a comparison theorem for reflected BSDE of the form \eqref{app_other_BSDE_equation_to_approx}.

\begin{teo}\label{otherBSDE_comparison}
	Assume \ref{ass:rhosbound} holds for $\rho$ and $A$. We are given two sets of data $\xi^i,f^i,g^i,h^i$ satisfying assumptions at the beginning of the section for some $\beta>L'$, and such that $\xi^1\leq \xi^2$, $f_s^1\leq f_s^2$, $g_s^1\leq g_s^2$ and $h_s^1\leq h_s^2$. Then we have that the respective solutions $(Y^i,U^i,Z^i,K^i)$ are such $Y_t^1\leq Y_t^2$ for all $t\in[0,T]$, $\mathbb{P}$-almost surely.
\end{teo}

\begin{proof}
	We can approximate each RBSDE with the BSDE defined above. That is for $i=1,2$:
	\begin{multline*}
	Y^{i,n}_t=\xi^i+\int_t^Tf_s^idA_s+\int_t^T g_s^i ds+\int_t^T\int_EU_s^{i,n}(\rho_s(e)-1)\phi_s(de)dA_s\\-\int_t^T U^{i,n}_s(e)q(dsde)-\int_t^T Z_s^{i,n}+K_T^{i,n}-K_t^{i,n}
	\end{multline*}
	with $K_t^{i,n}=n\int_0^t(Y_s^{i,n}-h_s^i)^-ds$. Remember that $Y_t^{i,n}\nearrow Y_t^i$ thanks to proposition \ref{otherBSDE_approximation_final}. Since $h_s^1\leq h_s^2$, we have that $(y-h_s^1)^-\leq (y-h_s^2)^-$ and we can apply the comparison theorem \ref{appendix_BSDE_teo_general_comparison_simple_BSDE} and obtain that
	$$
	Y_t^{1,n}\leq Y_t^{2,n} \quad\forall t\in[0,T],\; \mathbb{P}\text{-a.s. } \forall n\geq 1.
	$$
	By taking the limit for $n\rightarrow+\infty$, we obtain that $Y_t^{1}\leq Y_t^{2} \quad\forall t\in[0,T],\; \mathbb{P}\text{-a.s. }$.
\end{proof}

\printbibliography
\end{document}